\documentclass[11pt,reqno]{amsart}

\usepackage{amssymb,amscd,amsxtra,soul,enumerate,hyperref,esint}
\hypersetup{colorlinks=false}

\numberwithin{equation}{section}

\theoremstyle{plain}

\newtheorem{thm}{Theorem}[section]
\newtheorem{lem}[thm]{Lemma}

\newtheorem{prop}[thm]{Proposition}

\theoremstyle{definition}

\newtheorem{defn}[thm]{Definition}

\newtheorem{ex}[thm]{Example}

\theoremstyle{remark}

\newtheorem{rem}[thm]{Remark}
\newtheorem{claim}{Claim}

\newtheorem*{ack}{Acknowledgment}

\newcommand{\R}{\mathbb{R}}

\newcommand{\Z}{\mathbb{Z}}

\newcommand{\FF}{\mathcal{F}}
\newcommand{\GG}{\mathcal{G}}
\newcommand{\HH}{\mathcal{H}}

\newcommand{\PP}{\mathcal{P}}

\newcommand{\fg}{\mathfrak{g}}

\newcommand{\fX}{\mathfrak{X}}

\newcommand{\bfh}{\mathbf{h}}

\newcommand{\bfM}{\mathbf{M}}

\newcommand{\bfT}{\mathbf{T}}

\newcommand{\bfFF}{\boldsymbol{\FF}}

\newcommand{\bfpi}{\boldsymbol{\pi}}

\newcommand{\supp}{\operatorname{supp}}

\newcommand{\im}{\operatorname{im}}

\newcommand{\id}{\operatorname{id}}

\newcommand{\Aut}{\operatorname{Aut}}

\newcommand{\Diffeo}{\operatorname{Diffeo}}

\newcommand{\Hol}{\operatorname{Hol}}

\newcommand{\sign}{\operatorname{sign}}
\newcommand{\Fix}{\operatorname{Fix}}
\newcommand{\rank}{\operatorname{rank}}

\newcommand{\dom}{\operatorname{dom}}

\newcommand{\ev}{\operatorname{ev}}

\newcommand{\Aff}{\operatorname{Aff}}
\newcommand{\SL}{\operatorname{SL}}
\newcommand{\PSL}{\operatorname{PSL}}
\newcommand{\SO}{\operatorname{SO}}
\newcommand{\PSO}{\operatorname{PSO}}

\newcommand{\Hess}{\operatorname{Hess}}

\newcommand{\Cinftyc}{C^\infty_{\text{\rm c}}}

\newcommand{\sm}{\smallsetminus}

\newcommand{\olfX}{\overline{\fX}}
\newcommand{\fXcom}{\fX_{\text{\rm com}}}
\newcommand{\olfXcom}{\overline{\fX}_{\text{\rm com}}}

\usepackage{color}
\definecolor{darkgreen}{cmyk}{1,0,1,.2}
\definecolor{m}{rgb}{1,0.1,1}

\newdimen\theight
\def\TeXref#1{%
             \leavevmode\vadjust{\setbox0=\hbox{{\tt
                     \quad\quad  {\small \textrm #1}}}%
             \theight=\ht0
             \advance\theight by \lineskip
             \kern -\theight \vbox to
             \theight{\rightline{\rlap{\box0}}%
             \vss}%
             }}%

\title{Simple foliated flows}

\author[J.A. \'Alvarez L\'opez]{Jes\'us A. \'Alvarez L\'opez}
\address{Department/Institute of Mathematics\\
         University of Santiago de Compostela\\
         15782 Santiago de Compostela\\ Spain}
\email{jesus.alvarez@usc.es}

\author[Y.A. Kordyukov]{Yuri A. Kordyukov}
\address{Institute of Mathematics\\ 
Ufa Federal Research Centre\\ 
Russian Academy of Science\\ 
112 Chernyshevsky str.\\ 
450008 Ufa\\ Russia}
\email{yurikor@matem.anrb.ru}

\author[E. Leichtnam]{Eric Leichtnam}
\address{Institut de Math\'ematiques de Jussieu-PRG\\ CNRS\\ Batiment Sophie Germain (bureau 740)\\ Case~7012\\ 75205 Paris Cedex 13, France}
\email{ericleichtnam@math.jussieu.fr}

\thanks{The authors are partially supported by FEDER/Ministerio de Ciencia, Innovaci\'on y Universidades/AEI/MTM2017-89686-P and MTM2014-56950-P, and Xunta de Galicia/2015 GPC GI-1574 and ED431C 2019/10 with FEDER funds.}

\date{\today}



\begin{document}

\maketitle

\begin{abstract}
We describe transversely oriented foliations of codimension one on closed manifolds that admit simple foliated flows.
\end{abstract}

\tableofcontents

\section{Introduction}

In this paper, we describe transversely oriented foliations of codimension one on closed manifolds that admit simple foliated flows. Our motivation to study simple foliated flows comes from the role that they play in Deninger's program \cite{Deninger1998,Deninger2001,Deninger2002,Deninger-Arith_geom_anal_fol_sps,Deninger2008}. These are exactly those foliated flows for which a dynamical Lefschetz trace formula conjectured by Deninger holds. For the study of the associated Lefschetz trace formula, we refer to \cite{AlvKordy2001,AlvKordy2002,AlvKordy2008a,AlvKordyLeichtnam-aorfobg,AlvKordyLeichtnam-atffff}. A related classification of foliated dynamical systems was given in \cite{KimMorishitaNodaTerashima-FDS}.

Let $\FF$ be a smooth foliation of codimension one on a closed manifold $M$. Flows on $M$ are foliated when they map leaves to leaves. This means that their infinitesimal generators are infinitesimal transformations of $(M,\FF)$. These infinitesimal transformations form the normalizer $\fX(M,\FF)$ of the Lie subalgebra $\fX(\FF)\subset\fX(M)$ of vector fields tangent to the leaves, obtaining the quotient Lie algebra $\olfX(M,\FF)=\fX(M,\FF)/\fX(\FF)$. The elements of $\olfX(M,\FF)$, called transverse vector fields, can be considered as leafwise invariant sections of the normal bundle of $\FF$. 

Let $(\Sigma,\HH)$ be the holonomy pseudogroup of $\FF$. The infinitesimal generators of $\HH$-equivariant local flows on $\Sigma$ are the $\HH$-invariant vector fields. These invariant vector fields form a Lie subalgebra $\fX(\Sigma,\HH)\subset\fX(\Sigma)$. There is a canonical identity $\olfX(M,\FF)\equiv\fX(\Sigma,\HH)$, and every foliated flow $\phi$ induces an $\HH$-equivariant local flow $\bar\phi$ on $\Sigma$.

Simple fixed points and simple closed orbits of a flow $\phi$ can be defined by using a transversality condition between the graph of $\phi$ and the diagonal. The flow is simple when all of its fixed points and closed orbits are simple. Using the canonical identity between leaf and orbit spaces, $M/\FF\equiv\Sigma/\HH$, the leaves preserved by a foliated flow $\phi$, which will be shortly called preserved leaves in the sequel, correspond to $\HH$-orbits consisting of fixed points of $\bar\phi$. A preserved leaf $L$ is called transversely simple if the corresponding fixed points $\bar p$ of $\bar\phi$ are simple. In this case, $\bar\phi^t_*=e^{\varkappa t}$ on $T_{\bar p}\Sigma\equiv\R$ for some $\varkappa=\varkappa_L\in\R^\times:=\R\sm\{0\}$, which depends only on $L$. It is said that $\phi$ is transversely simple when all of its preserved leaves are transversely simple. Clearly, every simple flow is transversely simple. 

Let $L$ be any compact leaf whose holonomy group $\Hol L$ can be described by germs of homotheties at $0$. This description of $\Hol L$ can be achieved with a foliated chart $(U,(x,y))$ around any point of $L$, where $x$ is the transverse coordinate. The same kind of description of $\Hol L$ is given by the foliated chart $(U,(u,y))$, with $u=x\,|x|^{\alpha-1}$ ($0<\alpha\ne1$), which is not smooth at $U\cap L$. A transverse power change of the differentiable structure around $L$ is defined by requiring all of these new charts to be smooth. In Sections~\ref{ss: change of diff struct} and~\ref{ss: tubular neighborhoods}, we give a description of this new differential structure in terms of a defining form of $\FF$ and a defining function of $L$ on some tubular neighborhood.

The following is our main result, which is part of Theorem~\ref{t: simple foliated flows}.

\begin{thm}\label{t: description}
Let $\FF$ be a transversely oriented smooth foliation of codimension one on a closed manifold $M$. Then $\FF$ admits a (transversely) simple foliated flow in the following cases and uniquely in these ones:
\begin{enumerate}[(i)]

\item\label{i: fiber bundle, description}  $\FF$ is a fiber bundle over $S^1$ with connected fibers.

\item\label{i: minimal Lie foln, description} $\FF$ is a minimal $\R$-Lie foliation.

\item\label{i: transv affine, description} $\FF$ is an elementary transversely affine foliation whose developing map is surjective over $\R,$ and whose global holonomy group is a non-trivial group of homotheties.

\item\label{i: transv proj, description} $\FF$ is a transversely projective foliation whose developing map is surjective over the real projective line $S^1_\infty=\R\cup\{\infty\}$, and whose global holonomy group consists of the identity and hyperbolic elements with a common fixed point set.

\item\label{i: transv change of diff str, description} $\FF$ is obtained from~(\ref{i: transv affine, description}) or~(\ref{i: transv proj, description}) using transverse power changes of the differentiable structure of $M$ around the compact leaves.
\end{enumerate}
\end{thm}

In all cases of Theorem~\ref{t: description}, $\FF$ is almost without holonomy.

In the cases~(\ref{i: fiber bundle, description}) and~(\ref{i: minimal Lie foln, description}), $\FF$ is defined by a non-vanishing closed form $\omega$ of degree one, and therefore it is indeed without holonomy. The group of periods of $[\omega]\in H^1(M)$ has rank $1$ in~(\ref{i: fiber bundle, description}), and rank $>1$ in~(\ref{i: minimal Lie foln, description}).

In the case~(\ref{i: fiber bundle, description}), all leaves are compact and we have $\olfX(M,\FF)\equiv\fX(S^1)$. Moreover, for any even number of points, $x_1,\dots,x_{2m}\in S^1$ ($m\ge0$), in cyclic order, and numbers $\varkappa_1,\dots,\varkappa_{2m}\in\R^\times$, with alternate sign, there is some (transversely) simple foliated flow $\phi$ whose preserved leaves are the fibers $L_i$ over the points $x_i$, with $\varkappa_{L_i}=\varkappa_i$.  If $m>0$, then $\phi$ has no closed orbits transverse to the leaves. If $m=0$, then $\phi$ has no preserved leaves, and therefore no fixed points. Every transversely simple foliated flow is of this form.

In the cases~(\ref{i: minimal Lie foln, description})--(\ref{i: transv proj, description}), $\olfX(M,\FF)$ is of dimension one.

In the case~(\ref{i: minimal Lie foln, description}), $\olfX(M,\FF)$ is generated by a non-vanishing transverse vector field, and the transversely simple foliated flows have no preserved leaves.

In the cases~(\ref{i: transv affine, description})--(\ref{i: transv change of diff str, description}), there is a finite number of compact leaves, which are the preserved leaves of every transversely simple foliated flow.

In the case~(\ref{i: transv affine, description}) or~(\ref{i: transv proj, description}), for every transversely simple flow $\phi$, there is some $\varkappa\in\R^\times$ such that the set of numbers $\varkappa_L$ is $\{\varkappa\}$ or $\{\pm\varkappa\}$, respectively.

In the cases~(\ref{i: transv affine, description}) and~(\ref{i: transv proj, description}), the holonomy groups of the compact leaves can be described by germs of homotheties at $0$. Thus transverse power changes of the differentiable structure can be considered around them to get the case~(\ref{i: transv change of diff str, description}). $\olfX(M,\FF)$ and the (transversely) simple foliated flows are independent of these changes of the differentiable structure. But every $|\varkappa_L|$ can be modified arbitrarily by performing such changes, keeping $\sign(\varkappa_L)$ invariant.

\begin{ack}
We thank Hiraku Nozawa for helpful discussions about the contents of this paper.
\end{ack}

\section{Preliminaries}\label{s: prelims}

Let $M$ be a (smooth) manifold of dimension $n$.

\subsection{Simple flows}\label{ss: simple flows}

Let $Z\in\fX(M)$ with local flow $\phi:\Omega\to M$, where $\Omega$ is an open neighborhood of $M\times\{0\}$ in $M\times\R$. For $p\in M$ and $t\in\R$, let
\[
\Omega_p=\{\,\tau\in\R\mid(p,\tau)\in\Omega\,\}\;,\quad\Omega^t=\{\,q\in M\mid(q,t)\in\Omega\,\}\;,
\]
and let $\phi^t=\phi(\cdot,t):\Omega^t\to M$. It is said that $p\in M$ is a \emph{fixed point} of $\phi$ if it is a fixed point of $\phi^t$ for all $t$ in some neighborhood of $0$ in $\Omega_p$; in other words, if $Z(p)=0$. The fixed point set is denoted by $\Fix(\phi)$. For every $p\in\Fix(\phi)$, there is an endomorphism $H_p$ of $T_pM$ so that $\phi^t_*=e^{tH_p}$ on $T_pM$. Then $p$ is called \emph{simple}\footnote{The terms \emph{transverse}/\emph{elementary} are also used instead of simple/generic.} (respectively, \emph{generic}) if $H_p$ is an automorphism (respectively, no eigenvalue of $H_p$ has zero real part).

Now assume that $Z$ is complete with flow $\phi:M\times\R\to M$, which may considered as a one-parameter subgroup of diffeomorphisms, $\phi=\{\phi^t\}\subset\Diffeo(M)$. On $M\sm\Fix(\phi)$, let $N\phi$ denote the normal bundle to the orbits of $\phi$; i.e., $N_p\phi=T_pM/\R\,Z(p)$ for all $p\in M\sm\Fix(\phi)$. For every closed orbit $c$ of $\phi$ (without including fixed points), let $\ell(c)$ denote its smallest positive period. Recall that $c$ is called \emph{simple} (respectively, \emph{generic}) if the eigenvalues of the isomorphism of $N_p\phi$ induced by $\phi^{\ell(c)}_*$ are different from $1$ (respectively, have modulo different from $1$) for all $p\in c$.

It is said that $\phi$ (or $Z$) is \emph{simple} if all of its fixed points and closed orbits are simple. This means that the maps $M\times\R^\pm\to M^2\times\R^\pm$, $(p,t)\mapsto(p,\phi^t(p),t)$ and $(p,t)\mapsto(p,p,t)$, are transverse \cite[Lecture~2, Lemma~7]{Guillemin1977}. Thus fixed points and closed orbits are isolated in this case;  there are finitely many of them if $M$ is compact. 

On the other hand, $\phi$ (or $Z$) is called \emph{generic} if all of its fixed points and closed orbits are generic, and their stable and unstable manifolds are transverse---the definition of the stable and unstable manifolds is omitted because we will not use them. A theorem of Kupka \cite{Kupka1963,Kupka1964} and Smale \cite{Smale1963} states that, for any closed manifold $M$, the set of generic smooth vector fields on $M$ is residual in $\fX(M)$ with the $C^\infty$ topology (see also \cite{Peixoto1962} for the case of closed surfaces). This was generalized to open manifolds by Peixoto \cite{Peixoto1967}, using the strong $C^\infty$ topology.

\begin{rem}\label{r: fZ}
Suppose that $M$ is closed. For $0<f\in C^\infty(M)$, let $Z'=fZ\in\fX(M)$. The flow $\phi'$ of $Z'$ has the same orbits as $\phi$, considered as sets, but with possibly different time parameterizations; precisely, there is a smooth function $t':M\times\R\to\R$ such that $\phi(p,t)=\phi'(p,t'(p,t))$ for all $(p,t)$. It easily follows that $\phi'$ is simple if and only if $\phi$ is simple.
\end{rem}

\begin{ex}\label{ex: Morse}
Suppose that $M$ is closed, and let $f$ be a Morse function on $M$. For any Riemannian metric on $M$, the flow $\phi$ of $\nabla f$ has no closed orbits because $f$ is strictly increasing on every orbit in $M\sm\Fix(\phi)$. Moreover every $p\in\Fix(\phi)$ is generic because $H_p$ is given by $\Hess f(p)$, whose eigenvalues are in $\R^\times$. The transversality of the stable and unstable manifolds of all fixed points holds for an open dense set of Riemannian metrics in the $C^2$ topology \cite[Section~2.3]{Schwarz1993} (see also \cite{Smale1961}). In this case, $\nabla f$ is generic without closed orbits.
\end{ex}

\subsection{Collar and tubular neighborhoods}\label{ss: collar}

Suppose that $M$ is compact with boundary, and let $\mathring M$ denote its interior. There exists a \emph{boundary defining function} $x\in C^\infty(M)$, in the sense that $x\ge0$, $x^{-1}(0)=\partial M$, and $dx\ne0$ on $\partial M$. Then an (open) \emph{collar neighborhood} of the boundary, $\varpi:T\to\partial M$, can be chosen of the form\footnote{In a product, the projections may be indicated as subindexes of the factors.} $T\equiv[0,\epsilon)_x\times\partial M_\varpi$ for some $\epsilon>0$. For any chart $(V,y)$ of $\partial M$, we get a chart $(U\equiv[0,\epsilon)_x\times V,(x,y))$ of $M$ adapted to $\partial M$.

Now assume that $M$ is closed. Let $M^0\subset M$ be a (possibly disconnected) regular and transversely oriented submanifold of codimension one, and let $M^1=M\sm M^0$.  Since $M^0$ is transversely oriented, there is a \emph{defining function} $x$ of $M^0$ in some open $W\subset M$, in the sense that $x\in C^\infty(W)$, $M^0=x^{-1}(0)\subset W$, and $dx\ne0$ on $M^0$. Then there is an (open) \emph{tubular neighborhood} of $M^0$ in $W$, $\varpi:T\to M^0$, of the form $T\equiv(-\epsilon,\epsilon)_x\times M^0_\varpi$ for some $\epsilon>0$. For any chart $(V,y)$ of $M^0$, we get a chart $(U\equiv(-\epsilon,\epsilon)_x\times V,(x,y))$ of $M$ adapted to $M^0$. Let $\bfM$ be the manifold with boundary defined by ``cutting'' $M$ along $M^0$; i.e., modifying $M$ only on the tubular neighborhood $T\equiv(-\epsilon,\epsilon)\times M^0$, which is replaced with $\bfT=((-\epsilon,0]\sqcup[0,\epsilon))\times M^0$ in the obvious sense. Thus $\partial\bfM\equiv M^0\sqcup M^0$, and $\mathring\bfM\equiv M^1$. There is a canonical projection $\bfpi:\bfM\to M$, which is the combination of the identity on $\mathring\bfM\equiv M^1$ and the map $\bfT\to T$ induced by the canonical projection $(-\epsilon,0]\sqcup[0,\epsilon)\to(-\epsilon,\epsilon)$. This projection realizes $M$ as a quotient space of $\bfM$ by ``gluing'' the two copies of $M^0$ in the boundary. 

The connected components of $\bfM$ can be also described as the metric completion of the connected components of $M^1$ with respect to the restriction of any Riemannian metric on $M$, and then $\bfpi$ is given by taking limits of Cauchy sequences.

\subsection{Foliations}\label{ss: folns}

The concepts used here are explained in standard references on foliations, like \cite{Haefliger1962,HectorHirsch1981-A,HectorHirsch1983-B,CamachoLinsNeto1985,Godbillon1991,Tamura1992,CandelConlon2000-I,CandelConlon2003-II,Walczak2004}. Let $\FF$ be a (smooth) \emph{foliation}\footnote{It is also said that $(M,\FF)$ is a \emph{foliated manifold}.} on $M$ of codimension $n'$ and dimension $n''$. Locally, $\FF$ can be described by a (smooth) \emph{foliated chart} $(U,x)$, where $x=(x',x''):U\to x(U)=\Sigma\times B''$ for open balls, $\Sigma$ in $\R^{n'}$ and $B''$ in $\R^{n''}$. In the case of codimension one, we may use the notation $(x,y)$ instead of $(x',x'')$. The fibers of $x'$ are the \emph{plaques}. The intersections of plaques of different foliated charts are open in the plaques. Thus all plaques of all foliated charts form a base of a finer topology on $M$ whose path-connected components are the  \emph{leaves}, which are injectively immersed $n''$-submanifolds. The leaf through any $p\in M$ may be denoted by $L_p$. The submanifolds transverse to the leaves are called \emph{transversals}; for example, the fibers of the maps $x''$ are local transversals. A transversal is called \emph{complete} when it meets all leaves. A \emph{foliated atlas} is a covering of $M$ by foliated charts.

If a smooth map $\phi:M'\to M$ transverse to (the leaves of) $\FF$, then the connected components of the inverse images of the leaves of $\FF$ are the leaves of the \emph{pull-back} $\phi^*\FF$, which is a smooth foliation on $M'$ of codimension $n'$. For the inclusion map of any open $U\subset M$, this defines the \emph{restriction} $\FF|_U$.

Foliations on manifolds with boundary can be similarly defined, with leaves tangent or transverse to the boundary. The concepts and properties of foliations considered here have obvious versions with boundary.

\subsection{Holonomy}\label{ss: holonomy}

Let $\{U_k,x_k\}$ be a foliated atlas of $\FF$ with $x_k=(x'_k,x''_k)$ and $x_k(U_k)=\Sigma_k\times B''_k$. Assume that it is \emph{regular} in the following sense: $\{U_k\}$ is locally finite, there are foliated charts $(V_k,y_k)$ with $\overline{U_k}\subset V_k$ and $y_k|_{U_k}=x_k$, and $U_k\cup U_l$ is in the domain of some foliated chart if $U_k\cap U_l\ne\emptyset$. Then, with the notation $\Sigma_{kl}=x'_k(U_k\cap U_l)$, the \emph{elementary holonomy transformations} $h_{kl}:\Sigma_{lk}\to\Sigma_{kl}$ are defined by $h_{kl}x'_l=x'_k$ on $U_k\cap U_l$. Let $\HH$ denote the representative of the \emph{holonomy pseudogroup} on $\Sigma:=\bigsqcup_k\Sigma_k$ generated by the local transformations $h_{kl}$. The $\HH$-orbit of every $\bar p\in\Sigma$ is denoted by $\HH(\bar p)$. The maps $x'_k$ define a homeomorphism between the leaf space $M/\FF$ and the orbit space $\Sigma/\HH$.

Let $c:I:=[0,1]\to L$ be a path in a leaf from $p\in L\cap U_k$ to $q\in L\cap U_l$, and let $\bar p=x'_k(p)\in\Sigma_k$ and $\bar q=x'_l(q)\in\Sigma_l$. Take a partition of $I$, $0=t_0<t_1<\dots<t_m=1$, and a sequence of indices, $k=k_1,k_2,\dots,k_m=l$, such that $c([t_{i-1},t_i])\subset U_{k_i}$ for $i=1,\dots,m$. Let $h_c=h_{k_mk_{m-1}}\cdots h_{k_2k_1}$. We have $\bar p\in\dom h_c\subset\Sigma_k$ and $\bar q=h_c(\bar p)\in\im h_c\subset\Sigma_l$. The germ $\bfh_c$ of $h_c$ at $\bar p$ is the (\emph{germinal}) \emph{holonomy} of $c$, and the tangent map $h_{c*}:T_{\bar p}\Sigma_k\to T_{\bar q}\Sigma_l$ is its \emph{infinitesimal holonomy}. End-point homotopic paths in $L$ define the same holonomy. Thus, taking $p=q$ and $k=l$, we get the \emph{holonomy homomorphism} onto the \emph{holonomy group}, $\bfh=\bfh_L:\pi_1L=\pi_1(L,p)\to\Hol L=\Hol(L,p)$, $[c]\mapsto\bfh_c$, which is independent of the foliated chart containing $p$ up to conjugation. The \emph{holonomy cover} $\widetilde L=\widetilde L^{\text{\rm hol}}$ of $L$ is defined by $\pi_1\widetilde L=\ker\bfh_L$. If\footnote{In abstract groups, the identity element is denoted by $e$.} $\Hol L=\{e\}$, it is said that $L$ has \emph{no holonomy}. The union of leaves without holonomy is a dense $G_\delta$ subset \cite{Hector1977a,EpsteinMillettTischler1977}. If all leaves have no holonomy, then $\FF$ is said to be \emph{without holonomy}. According to Reeb's local stability, if $L$ is compact, then the germ of $\FF$ at $L$ is determined by $\bfh_L$ using a construction called \emph{suspension} \cite[Section~2.7]{Haefliger1962} (see also \cite[Theorem~2.1.7]{HectorHirsch1981-A}, \cite[Theorem~IV.2]{CamachoLinsNeto1985}, \cite[Theorem~II.2.29]{Godbillon1991}, \cite[Theorem~2.3.9]{CandelConlon2000-I}). Similarly, we have the concepts of \emph{infinitesimal holonomy groups} of the leaves, and leaves/foliations \emph{without infinitesimal holonomy}. 

With the above notation, an element of $\Hol L$ is called \emph{quasi-analytic} if, either it is the identity, or it is represented by some local transformation $h$ such that $h|_V\ne\id_V$ for all open $V\subset\dom h$ with $\bar p\in\overline V$. $\Hol L$ is called \emph{quasi-analytic} when all of its elements are quasi-analytic.

In the case of codimension one, $\Hol L$ can be described by germs at $0$ of local transformations of $\R$. Then $\FF$ is said to be \emph{infinitesimally $C^\infty$-trivial} at $L$ if $h'(0)=1$ and $h^{(k)}(0)=0$ ($k>1$) for all local transformation $h$ representing an element of $\Hol L$. For instance, this property is satisfied if $\Hol L$ is generated by non-quasi-analytic elements.

\subsection{Infinitesimal transformations and transverse vector fields}\label{ss: infinitesimal transfs}

Let $T\FF\subset TM$ denote the subbundle of vectors tangent to the leaves, and let $N\FF=TM/T\FF$. The terms \emph{leafwise}\footnote{The terms ``\emph{tangent}'' or ``\emph{vertical}'' are also used instead of ``leafwise''.}/\emph{normal} are used for these vector bundles, their elements and smooth sections (vector fields). The leafwise vector fields form a Lie subalgebra and $C^\infty(M)$-submodule, $\fX(\FF)\subset\fX(M)$. Its normalizer is the Lie algebra $\fX(M,\FF)$ of \emph{infinitesimal transformations} of $(M,\FF)$, and $\olfX(M,\FF)=\fX(M,\FF)/\fX(\FF)$ is the Lie algebra of \emph{transverse vector fields}. An \emph{orientation} (respectively,  \emph{transverse orientation}) of $\FF$ is an orientation of the vector bundle $T\FF$ (respectively, $N\FF$).

For any $X$ in $TM$ (respectively, $\fX(M)$ or $\fX(M,\FF)$), let $\overline{X}$ denote the induced element of\footnote{The space of smooth sections of a vector bundle $E$ is denoted by $C^\infty(M;E)$.} $N\FF$ (respectively, $C^\infty(M;N\FF)$ or $\olfX(M,\FF)$). $N\FF$ becomes a \emph{leafwise flat} vector bundle with the canonical flat $T\FF$-partial connection $\nabla^\FF$ given by $\nabla^\FF_V\overline{X}=\overline{[V,X]}$ for $V\in\fX(\FF)$ and $X\in\fX(M)$. The leafwise parallel transport along any piecewise smooth path $c$ is the infinitesimal holonomy $h_{c*}:T_{\bar p}\Sigma_k\equiv N_p\FF\to T_{\bar q}\Sigma_l\equiv N_q\FF$. 

$\olfX(M,\FF)$ can be realized as the linear subspace of $C^\infty(M;N\FF)$ consisting of leafwise flat normal vector fields. The local projections $x'_k$ induce a canonical isomorphism of $\olfX(M,\FF)$ to the Lie algebra $\fX(\Sigma,\HH)$ of $\HH$-invariant tangent vector fields on $\Sigma$. The notation $\overline X$ is also used for the element of $\fX(\Sigma,\HH)$ that corresponds to $X\in\olfX(M,\FF)$.

When $M$ is not closed, let $\fXcom(\FF)\subset\fX(\FF)$ and $\fXcom(M,\FF)\subset\fX(M,\FF)$ denote the subsets of complete vector fields, and $\olfXcom(M,\FF)\subset\olfX(M,\FF)$ the projection of $\fXcom(M,\FF)$.

\subsection{Foliated maps and foliated flows}\label{ss: fol maps}

A (smooth) map between foliated manifolds, $\phi:(M_1,\FF_1)\to(M_2,\FF_2)$, is called \emph{foliated} if it maps leaves to leaves. Then its tangent map defines morphisms, $\phi_*:T\FF_1\to T\FF_2$ and $\phi_*:N\FF_1\to N\FF_2$, the second one being compatible with the leafwise flat structures.

Let $\Diffeo(M,\FF)\subset\Diffeo(M)$ be the subgroup of foliated diffeomorphisms. A smooth flow $\phi$ on $M$ is called \emph{foliated} if $\phi^t\in\Diffeo(M,\FF)$ for all $t$. This concept can be extended to a local flow $\phi:\Omega\to M$ by considering the restriction to $\Omega$ of the foliation on $M\times\R$ with leaves $L\times\{t\}$, for leaves $L$ of $\FF$ and points $t\in\R$. For $X\in\fX(M)$ (respectively, $X\in\fXcom(M)$), we have $X\in\fX(M,\FF)$ (respectively, $X\in\fXcom(M,\FF)$) if and only if its local flow (respectively, flow) is foliated.

For $X\in\fXcom(M,\FF)$ with foliated flow $\phi$, let $\bar\phi$ be the local flow on $\Sigma$ generated by $\overline X\in\fX(\Sigma,\HH)$, which  corresponds to $\phi$ via the maps $x'_k$. In an obvious sense, $\bar\phi$ is $\HH$-equivariant, and therefore it defines an $\HH$-equivariant local flow $\bar\phi$ on any other representative of the holonomy pseudogroup.

\subsection{Riemannian foliations}\label{ss: Riem folns}

The $\HH$-invariant structures on $\Sigma$ are called (\emph{invariant}) \emph{transverse structures}. A transverse orientation has this interpretation. Other examples are \emph{transverse Riemannian metrics} and \emph{transverse parallelisms}. Their existence defines the classes of (\emph{transversely}) \emph{Riemannian} and \emph{transversely parallelizable} (\emph{TP}) foliations. A Lie subalgebra $\fg\subset\fX(\Sigma,\HH)$ generated by a transverse parallelism is called a \emph{transverse Lie structure}, giving rise to the concept of ($\fg$-)\emph{Lie foliation}. 

Let $G$ be the simply connected Lie group with Lie algebra $\fg$. $\FF$ is a $\fg$-Lie foliation just when $\{U_k,x_k\}$ can be chosen so that every $\Sigma_k$ is realized as an open subset of $G$ and the maps $h_{kl}$ are restrictions of left translations. 

Using the canonical isomorphism $\olfX(M,\FF)\cong\fX(\Sigma,\HH)$, a transverse parallelism can be given by a global frame of $N\FF$ consisting of transverse vector fields $\overline{X_1},\dots,\overline{X_{n'}}$. This frame defines a transverse Lie structure when it is a base of a Lie subalgebra $\fg\subset\olfX(M,\FF)$. If moreover $\overline{X_1},\dots,\overline{X_{n'}}\in\olfXcom(M,\FF)$, the TP or Lie foliation $\FF$ is called \emph{complete}. 

Similarly, a transverse Riemannian metric can be described as a leafwise flat Euclidean structure on $N\FF$. It is induced by a \emph{bundle-like metric} on $M$, in the sense that the maps $x'_k$ are Riemannian submersions.

It is said that $\FF$ is \emph{transitive at} $p\in M$ when the evaluation map $\ev_p:\fX(M,\FF)\to T_pM$ is surjective, or, equivalently, the evaluation map $\overline{\ev}_p:\olfX(M,\FF)\subset C^\infty(M;N\FF)\to N_p\FF$ is surjective. Similarly, $\FF$ is called \emph{transversely complete} (\emph{TC}) \emph{at} $p$ if $\ev_p(\fXcom(M,\FF))$ generates $T_pM$, or, equivalently, $\overline{\ev}_p(\olfXcom(M,\FF))$ generates $N_p\FF$. The transitive/TC point set is open and saturated. $\FF$ is called \emph{transitive}/\emph{TC} if it is transitive/TC at every point \cite[Section~4.5]{Molino1988}. 

TP foliations are transitive, and transitive foliations are Riemannian. In turn, Molino's theory describes Riemannian foliations in terms of TP foliations \cite{Molino1988}. A Riemannian foliation is called \emph{complete} if, using Molino's theory, the corresponding TP foliation is TC. Furthermore Molino's theory describes TC foliations in terms of complete Lie foliations with dense leaves. On the other hand, complete Lie foliations have the following description due to Fedida \cite{Fedida1971,Fedida1973} (see also \cite[Theorem~4.1 and Lemma~4.5]{Molino1988}). Assume that $M$ is connected and $\FF$ a complete $\fg$-Lie foliation. Let $G$ be the simply connected Lie group with Lie algebra $\fg$. Then there is a regular covering $\pi:\widetilde M\to M$, a fiber bundle $D:\widetilde M\to G$ (the \emph{developing} map) and a monomorphism\footnote{$\Aut(\pi)$ denotes the group of deck transformations of the covering $\pi:\widetilde M\to M$.} $h:\Gamma:=\Aut(\pi)\equiv\pi_1M/\pi_1\widetilde M\to G$ (the \emph{holonomy} homomorphism) such that the leaves of $\widetilde\FF:=\pi^*\FF$ are the fibers of $D$, and $D$ is $h$-equivariant with respect to the left action of $G$ on itself by left translations. As a consequence, $\pi$ restricts to diffeomorphisms between the leaves of $\widetilde\FF$ and $\FF$. The subgroup $\Hol\FF:=\im h\subset G$, isomorphic to $\Gamma$, is called the \emph{global holonomy group}. Since $D$ induces an identity $\widetilde M/\widetilde\FF\equiv G$, the $\pi$-lift and $D$-projection of vector fields define identities
\begin{equation}\label{olfX(M, FF) equiv ... equiv fX(G, Hol FF)}
\olfX(M,\FF)\equiv\olfX(\widetilde M,\widetilde\FF,\Gamma)\equiv\fX(G,\Hol\FF)\;,
\end{equation}
where a group within the parentheses to denote subspaces of invariant sections\footnote{This is preferred rather than the usual subindex to agree with $\fX(\Sigma,\HH)$ and $\fX(M,\FF)$.}. These identities give a precise realization of $\fg\subset\olfX(M,\FF)$ as the Lie algebra of left invariant vector fields on $G$. The holonomy pseudogroup of $\FF$ is equivalent to the pseudogroup on $G$ generated by the action of $\Hol\FF$ by left translations. Thus the leaves are dense if and only if $\Hol\FF$ is dense in $G$, which means $\fg=\olfX(M,\FF)$.

\subsection{Homogeneous foliations}\label{ss: homogeneous folns}

More generally, consider the homogeneous space $S=G/H$, defined by a closed subgroup of a connected Lie group, $H\subset G$. It is said that $\FF$ is a (\emph{transversely}) \emph{homogeneous} ($(G,S)$-) \emph{foliation} if $\{U_k,x_k\}$ can be chosen so that every $\Sigma_k$ is realized as an open subset of $S$ and the maps $h_{kl}$ are restrictions of the action of elements of $G$. In this case, there is a regular covering $\pi:\widetilde M\to M$, a smooth submersion $D:\widetilde M\to S$ and a monomorphism $h:\Gamma:=\Aut(\pi)\equiv\pi_1L/\pi_1\widetilde L\to G$ such that the leaves of $\widetilde\FF:=\pi^*\FF$ are the connected components of the fibers of $D$, and $D$ is $h$-equivariant \cite{Blumenthal1979} (see also \cite[Section~III.3]{Godbillon1991}). The terms of Fedida's description are also used in this case, as well as the notation $\Hol\FF=\im h$. This description is determined up to conjugation in $G$ in an obvious sense. Now $\widetilde M/\widetilde\FF$ is a possibly non-Hausdorff smooth manifold, and $D$ induces a local diffeomorphism $\overline D:\widetilde M/\widetilde\FF\to S$, which is $h$-equivariant with respect to the induced $\Gamma$-action on $\widetilde M/\widetilde\FF$. Like in~\eqref{olfX(M, FF) equiv ... equiv fX(G, Hol FF)}, we get  
\begin{equation}\label{olfX(M, FF) supset fX(im D, Hol FF)}
\olfX(M,\FF)\equiv\olfX(\widetilde M,\widetilde\FF,\Gamma)
\equiv\fX(\widetilde M/\widetilde\FF,\Gamma)\supset\fX(\im D,\Hol\FF)\;.
\end{equation}
The holonomy pseudogroup of $\FF$ is equivalent to the pseudogroup generated by the action of $\Gamma$ on $\widetilde M/\widetilde\FF$. In particular, for leaves, $L$ of $\FF$ and $\widetilde L$ of $\widetilde\FF$ with $\pi(\widetilde L)=L$ and $D(\widetilde L)=x\in S$, we have
\begin{equation}\label{Hol L < Hol_x FF}
\Hol L\equiv\{\,\gamma\in\Gamma\mid\gamma\cdot\widetilde L=\widetilde L\,\}
\cong h(\{\,\gamma\in\Gamma\mid\gamma\cdot\widetilde L=\widetilde L\,\})\subset\Hol_x\FF\;,
\end{equation}
where $\Hol_x\FF\subset\Hol\FF$ is the isotropy subgroup at $x$.

\section{Some classes of foliations of codimension one}\label{s: folns of codim 1}

\subsection{Preliminary considerations}\label{ss: prelim,  folns of codim 1}

Let $\FF$ be a smooth foliation of codimension one on a closed $n$-manifold $M$. Suppose that $\FF$ is transversely oriented, obtaining\footnote{We use the notation $\Lambda=\Lambda M=\bigwedge T^*M$.} $\omega,\theta\in C^\infty(M;\Lambda^1)$ such that $\omega$ defines\footnote{This means that $T\FF = \ker\omega$ and the transverse orientation is induced by $\omega$ on $N\FF$.} $\FF$ (with its transverse orientation) and $d\omega=\theta\wedge\omega$. There is some $X\in\fX(M)$ with $\omega(X)=1$; in fact, $\overline X\in C^\infty(M;N\FF)$ and $\omega$ determine each other. Note that $\FF$ is Riemannian just when $\omega$ can be chosen so that $d\omega=0$ ($\theta=0$); i.e., $X\in\fX(M,\FF)$. Actually, $\FF$ is an $\R$-Lie foliation in this case because $\R\,\overline{X}$ is a Lie subalgebra of $\olfX(M,\FF)$.

Take any leaf $L$ and $p\in L$, and a local transversal $\Sigma\equiv(-\epsilon,\epsilon)$ through $p\equiv0$ so that the transverse orientation corresponds to the standard orientation of $(-\epsilon,\epsilon)$. Since the holonomy maps defining the elements of $\Hol(L,p)$ preserve the orientation of $(-\epsilon,\epsilon)$, they can be restricted to $(-\epsilon,0]$ and $[0,\epsilon)$, defining the \emph{lateral holonomy groups} $\Hol_\pm(L,p)=\Hol_\pm L$. 

Recall that $L$ is said to be \emph{locally dense} if it is dense in some open saturated set. On the other hand, $L$ is said to be \emph{resilient} if there is some element of $\Hol(L,p)$, represented by some local diffeomorphism $f$ defined around $p$ in $\Sigma$, and there is some $q\ne p$ in $L\cap\dom f$ such that the sequence $f^k(q)$ is defined and converges to $p$.

Now a smooth connected closed transversal of $\FF$ is a smooth embedding $c:S^1\to M$ transverse to the leaves. It always has a (closed) tubular neighborhood $\varpi:T\to c(S^1)\equiv S^1$ in $M$, which can be chosen to be \emph{foliated} in the sense that its fibers are $(n-1)$ disks in the leaves. If $\FF$ is also oriented, then $\varpi$ trivial, $T\equiv S^1_\varpi\times D^{n-1}$, where $D^{n-1}$ is the standard disk in $\R^{n-1}$.

\subsection{$\R$-Lie foliations}\label{ss: R-Lie folns}

 Suppose that $\FF$ is a transversely complete $\R$-Lie foliation. This means that there is some $Z\in\fXcom(M,\FF)$ such that $\overline Z\ne0$ everywhere. Equivalently,  the orbits of the foliated flow $\phi$ of $Z$ are transverse to $\FF$. The Fedida's description of $\FF$ is given by a regular covering map $\pi:\widetilde M\to M$, a holonomy homomorphism $h:\Gamma:=\Aut(\pi)\to\R$, and the developing map $D:\widetilde M\to\R$ (Section~\ref{ss: Riem folns}). Thus $\Gamma\cong\im h\subset\R$ is abelian and torsion free. Let $\widetilde Z$ and $\tilde\phi$ be the lifts of $Z$ and $\phi$ to $\widetilde M$. Then $\widetilde Z$ is $\Gamma$-invariant and $D$-projectable. Without lost of generality, we can assume $D_*\widetilde Z=\partial_x\in\fX(\R)$, where $x$ denotes the standard global coordinate of $\R$. Thus $\tilde\phi$ is $\Gamma$-equivariant and induces via $D$ the flow $\bar\phi$ on $\R$ defined by $\bar\phi^t(x)=t+x$. This is the equivariant local flow induced by $\phi$ on this representative of the holonomy pseudogroup (Section~\ref{ss: Riem folns}). It is easy to check that $\phi^t$ preserves every leaf of $\FF$ if and only if $t\in\Hol\FF$.
 
 \begin{ex}\label{ex: Kronecker's flow}
 The simplest example of minimal $\R$-Lie foliation on a closed manifold is the Kronecker's flow on the torus $T^2\equiv\R^2/\Z^2$ \cite[Example~1.1.5]{CandelConlon2000-I}. It is induced by a foliation on $\R^2$ by parallel lines with irrational slope. This construction has an obvious generalization to higher dimensions, obtaining minimal $\R$-Lie foliations on every torus $T^n\equiv\R^n/\Z^n$ induced by foliations on $\R^n$ by appropriate parallel hyperplanes \cite[Example~1.1.8]{CandelConlon2000-I}.
 \end{ex}

\subsection{Foliations almost without holonomy}\label{ss: folns almost w/o hol}

Recall that $\FF$ is said to be \emph{almost without holonomy} when all non-compact leaves have no holonomy. The structure of such a foliation was described by Hector using the following \emph{model foliations} $\GG$ on compact manifolds $N$ (possibly with boundary) \cite[Structure Theorem]{Hector1972c}, \cite[Theorem~1]{Hector1978}:
\begin{enumerate}\addtocounter{enumi}{-1}

\item\label{i: model 0} $\GG$ is given by a trivial bundle over $[0,1]$,

\item\label{i: model 1} $\mathring\GG:=\GG|_{\mathring N}$ is given by a fiber bundle over $S^1$, or 

\item\label{i: model 2} all leaves of $\mathring\GG$ are dense in $\mathring N$.

\end{enumerate}
In the case where $\FF$ has finitely many leaves with holonomy, Hector's description is as follows. Let $M^0$ be the finite union of compact leaves with holonomy. Let $M^1=M\sm M^0$, whose connected components are denoted by $M^1_l$, with $l$ running in a finite index set, and let $\FF^1_l=\FF|_{M^1_l}$. For every $l$, there is a connected compact manifold\footnote{Since $M_l$ is the metric completion of $M^1_l$, the notation $\widehat M^1_l$ and $\widehat\FF^1_l$ would be more standard. But the notation $M_l$ is more appropriate for our use in \cite{AlvKordyLeichtnam-atffff} involving b-calculus.} $M_l$, possibly with boundary, endowed with a smooth transversely oriented foliation $\FF_l$ tangent to the boundary, sutisfying the following. Equipping $\bfM:=\bigsqcup_lM_l$ with the combination $\bfFF$ of the foliations $\FF_l$, there is a foliated smooth local embedding $\bfpi:(\bfM,\bfFF)\to(M,\FF)$, preserving the transverse orientations, so that $\bfpi:\mathring M_l\to M^1_l$ is a diffeomorphism for all $l$ (we may write $\mathring M_l\equiv M^1_l$), $\bfpi:\partial \bfM\to M^0$ is a $2$-fold covering map, and every $\FF_l$ is a model foliation. $M$ can be described by gluing the manifolds $M_l$ along corresponding pairs of boundary components. Equivalently, $\bfM$ can be described by cutting $M$ along $M^0$ (Section~\ref{ss: collar}).  Thus $\partial \bfM\equiv M^0\sqcup M^0$, and $\bfpi$ defines diffeomorphisms between corresponding connected components of $\partial\bfM$ and $M^0$.

\begin{rem}\label{r: Hector's description of folns almost w/o hol}
\begin{enumerate}[(i)]

\item\label{i: Hol_pm L} (See \cite[Lemma~7]{Hector1978} and its proof.) For indices $l_\pm$, and boundary leaves $L_\pm$ of $\FF_{l_\pm}$ with $L:=\bfpi(L_+)=\bfpi(L_-)$, we have $\Hol(L_\pm)\equiv\Hol_\pm L$. $\Hol_\pm L$ is the germ group at $0$ of a pseudogroup $\HH_{L,\pm}$ of local transformations of $\R^\pm\cup\{0\}$, generated by a (possibly empty) set of contractions and dilations defined around $0$. It follows that $\Hol_\pm L$ is an Archimedean totally ordered group, and therefore it is isomorphic to a subgroup of $(\R,+)$, obtaining that $\Hol L$ is abelian and torsion free. It is easy to see that the orbits of $\HH_{L,\pm}$ on $\R^\pm$ are singletons (respectively, monotone sequences with limit $0$, or dense) just when the rank of $\Hol_\pm L$ is $0$ (respectively, $1$, or $>1$). 

\item\label{i: the leaves of FF_l are compact} If $\FF_l$ is a model~(\ref{i: model 0}), or a model~(\ref{i: model 1}) with $\partial M_l=\emptyset$ ($M_l=M$ and $M^0=\emptyset$), then the leaves of $\FF_l$ are compact. 

\item\label{i: the leaves of mathring FF_l approach any leaf in the boundary} If $\FF_l$ is a model~(\ref{i: model 1}) with $\partial M_l\ne\emptyset$, or a model~(\ref{i: model 2}), then the leaves of $\mathring\FF_l$ are not compact. In fact, the whole of $\partial M_l$ is contained in the closure of every leaf of $\mathring\FF_l$. Hence, according to~(\ref{i: Hol_pm L}), the holonomy groups of the boundary leaves of $\FF_l$ are of rank $1$ (respectively, $>1$) if and only if $\FF_l$ is a model~(\ref{i: model 1}) with $\partial M_l\ne\emptyset$ (respectively, a model~(\ref{i: model 2})). 

\item\label{i: models (2) are Lie foliations} If $\FF_l$ is a model~(\ref{i: model 2}), then $\mathring\FF_l$ becomes a complete $\R$-Lie foliation after a possible change of the differentiable structure of $\mathring M_l$, keeping the same differentiable structure on the leaves \cite[Theorem~2]{Hector1978}. If moreover $\partial M_l=\emptyset$, then $\FF$ is homeomorphic to a minimal $\R$-Lie foliation.

\item\label{i: no resilient leaves} $\FF^1$ has no holonomy, and therefore $\FF$ has no resilient leaves. This holds because $\mathring\FF_l$ is given by a fiber bundle in the models~(\ref{i: model 0}) and~(\ref{i: model 1}), and is homeomorphic to a Lie foliation in the model~(\ref{i: model 2}) by~(\ref{i: models (2) are Lie foliations}).

\item\label{i: adding leaves to M^0} According to~(\ref{i: the leaves of FF_l are compact}) and~(\ref{i: the leaves of mathring FF_l approach any leaf in the boundary}), the description holds as well if $M^0$ is any finite union of compact leaves, including all leaves with holonomy. Thus, if $\FF_l$ is a model~(\ref{i: model 1}) with $\partial M_l=\emptyset$, then $M_l=M$ can be cut into models~(\ref{i: model 0}) by adding compact leaves to $M^0$. Conversely, if all foliations $\FF_l$ are models~(\ref{i: model 0}), then $\FF$ is a model~(\ref{i: model 1}) with $\partial M=\emptyset$.

\item\label{i: closed transversals} In the models~(\ref{i: model 1}) and~(\ref{i: model 2}), $\mathring\FF_l$ has smooth complete closed transversals (see \cite[Lemma~3.3.7]{CandelConlon2000-I}).

\end{enumerate}
\end{rem}

\begin{prop}\label{p: quasi-analytic => the same models}
If $\Hol L$ is quasi-analytic for all leaf $L\subset M^0$, then all foliations $\FF_l$ have the same model.
\end{prop}

\begin{proof}
For all leaves $L\subset M^0$, we have $\Hol_+L\cong\Hol_-L\cong\Hol L$ by the hypothesis on $\Hol L$. Then, by Remark~\ref{r: Hector's description of folns almost w/o hol}~(\ref{i: Hol_pm L})--(\ref{i: the leaves of mathring FF_l approach any leaf in the boundary}) and since $M$ is connected, the rank of the holonomy groups of all boundary leaves of all foliations $\FF_l$ is simultaneously $0$, $1$ or $>1$, and all foliations $\FF_l$ have the same model.
\end{proof}

\begin{ex}\label{ex: Reeb components}
A Reeb component on $D^{n-1}\times S^1$ is a model~(\ref{i: model 1}) \cite[Examples~1.1.12 and~3.3.11]{CandelConlon2000-I}, \cite[Example~I.3.14~(i)]{Godbillon1991}, \cite[Section~II.1.4.4]{HectorHirsch1981-A}. All of the Reeb components on $D^{n-1}\times S^1$ are homeomorphic, but they may not be diffeomorphic.

The Reeb components on $D^1\times S^1=[-1,1]\times S^1$ can be described as follows. Let $f:(-1,1)\to\R$ be a smooth function such that $|f^{(k)}(x)|\to\infty$ as $x\to\pm1$ for all order $k$. Then the graphs of the functions $f+c$ ($c\in\R$) are the interior leaves of a smooth foliation tangent to the boundary on the strip $[-1,1]\times\R$, which induces a smooth foliation $\GG$ on $[-1,1]\times S^1\equiv[-1,1]\times\R/\Z$. Its boundary leaves are $L_\pm=\{\pm1\}\times S^1$. The following examples of $f$ produce non-diffeomorphic foliations:
\begin{enumerate}[(i)]

\item\label{i: f(x) = exp 1/(1-x^2)} If $f(x)=\exp\frac{1}{1-x^2}$, then $\GG$ is infinitesimally $C^\infty$-trivial at $L_\pm$.

\item\label{i: f(x) = x^2/(1-x^2)} If $f(x)=\frac{x^2}{1-x^2}$, then $\GG$ is not infinitesimally $C^\infty$-trivial at $L_\pm$, but $L_\pm$ is without infinitesimal holonomy. 


\item\label{i: |f(x)| = ln(1-|x|)^mu} If $|f(x)|=\ln(1-|x|)^\mu$ ($\mu>0$) for $1-|x|$ small enough, then $\Hol L_\pm$ is generated by the germ of $u\mapsto e^{1/\mu}u$ at $0$ in $[0,\infty)$.

\end{enumerate}
\end{ex}

\begin{ex}\label{ex: tangential gluing of models}
Let $\GG_a$ ($a=1,2$) be transversely oriented models~(\ref{i: model 1}) or~(\ref{i: model 2}) of dimension $>1$ on manifolds $N_a$. If there is a diffeomorphism $\phi$ between boundary leaves, $L_a$ of $\GG_a$, then a tangential gluing via $\phi$ can be made, obtaining a foliation $\GG$ on $N:=N_1\cup_\phi N_2$, with the compact leaf $L:=L_1\cup_\phi L_2\subset\mathring N$ \cite[Section~3.4]{CandelConlon2000-I}, \cite[Example~I.3.14~(i)]{Godbillon1991}, \cite[Theorem~IV.4.2.2]{HectorHirsch1983-B}. $\GG$ may not be smooth. It is smooth only when, for all $\sigma\in\pi_1L_1$, the combination of representatives of $\bfh_\sigma$ and $\bfh_{\phi_*\sigma}$ are smooth maps (considering the elements of $\Hol L_1$ and $\Hol L_2$ as germs at $0$ of local transformations of $(-\infty,0]$ and $[0,\infty)$, respectively). For example, this is true if every $\bfh_\sigma$ and $\bfh_{\phi_*\sigma}$ are germs of homotheties at $0$ with the same ratio. This property is also guaranteed when every $\GG_a$ is infinitesimally $C^\infty$-trivial at $L_a$ \cite[Proposition~3.4.2]{CandelConlon2000-I}. 

We can continue making tangential gluing of models to produce a foliation $\FF$ on a closed manifold $M$. If every tangential gluing preserves smoothness, then $\FF$ is almost without holonomy with finitely many leaves with holonomy. The following are some examples of foliations obtained in this way:
\begin{enumerate}[(i)]

\item\label{i: Reeb foln} The Reeb foliation $\FF$ on $S^3$ is almost without holonomy and has one compact leaf $L$. It is obtained by tangential gluing of two Reeb components on $D^2\times S^1$, so that the gluing map interchanges meridian and longitude in the boundary leaves $S^1\times S^1$ \cite[Example~3.4.3 and Exercise~3.4.4]{CandelConlon2000-I}, \cite[Examples~I.3.14]{Godbillon1991}. Since $\Hol L$ has non-quasi-analytic generators, the Reeb components must be infinitesimally $C^\infty$-trivial at the boundary leaves to get smoothness of $\FF$.

\item\label{i: foln on S^n-1 times S^1} Let $\FF$ be foliation on $S^{n-1}\times S^1$ obtained by tangential gluing of two Reeb components on $D^{n-1}\times S^1$ using the identity map on the boundary leaves $S^{n-2}\times S^1$. $\FF$ becomes smooth if the Reeb components are infinitesimally $C^\infty$-trivial at the boundary leaves, but now this condition is not necessary to get smoothness (see Example~\ref{ex: trasv aff foln on S^n-1 times S^1} below).

\item\label{i: foln on T^2 or K} A smooth foliation $\FF$ on the $2$-torus or on the Klein bottle can be constructed by tangential gluing of Reeb components on $[-1,1]\times S^1$ of the type in Example~\ref{ex: Reeb components}~(\ref{i: |f(x)| = ln(1-|x|)^mu}), all of them constructed with the same constant $\mu$. The holonomy groups of the leaves with holonomy are generated by the germ of $u\mapsto e^{1/\mu}u$ at $0$ in $\R$.

\end{enumerate}
\end{ex}

\begin{ex}\label{ex: connected sum of folns} 
Let $\FF$ and $\GG$ be oriented and transversely orientable foliations of codimension one on closed $n$-manifolds $M$ and $N$ ($n\ge2$). Suppose that both of them are almost without holonomy, and that they have finitely many leaves with holonomy. Take smooth closed transversals, $c:S^1\to M^1$ of $\FF^1$ and $d:S^1\to N^1$ of $\GG^1$ (Remark~\ref{r: Hector's description of folns almost w/o hol}~(\ref{i: closed transversals})), and let $\FF'$ be the connected sum of $\FF$ and $\GG$ along $c$ and $d$ \cite[Example~I.2.20~(i)]{Godbillon1991}. 
$\FF'$ is another transversely orientable foliation almost without holonomy on a closed manifold, and it has finitely many leaves with holonomy.

For models~(\ref{i: model 1}) or~(\ref{i: model 2}), we can also consider their connected sum along smooth closed transversals in their interior. The result is a model~(\ref{i: model 1}) if both foliations are models~(\ref{i: model 1}), and a model~(\ref{i: model 2}) otherwise.
\end{ex}

\begin{ex}\label{ex: turbulization} 
Let $\FF$ be an oriented and transversely orientable foliation of codimension one on a closed $n$-manifold $M$. Suppose that $\FF$ is almost without holonomy, and that it has finitely many leaves with holonomy.  Let $(M',\FF')$ be the turbulization of $(M,\FF)$ along a smooth closed transversal $c:S^1\to M^1$ of $\FF^1$ \cite[Example~3.3.11]{CandelConlon2000-I}, \cite[Section~I.2.18]{Godbillon1991}. $\FF'$ is another transversely orientable foliation almost without holonomy, and it has finitely many leaves with holonomy. Actually, $\FF'$ can be considered as a connected sum along $c$ of $\FF$ and the foliation of Example~\ref{ex: tangential gluing of models}~(\ref{i: foln on S^n-1 times S^1}). 

The turbulization can be also applied to a model~(\ref{i: model 1}) or~(\ref{i: model 2}) along a smooth closed transversal in its interior. After removing the interior of the resulting Reeb component, we get a model of the same type.
\end{ex}

\subsection{Transversely affine foliations}\label{ss: transv affine folns}

Consider $\R$ as the homogeneous space defined by the canonical action of $\Aff^+(\R)$, the Lie group of its orientation preserving affine transformations. It is said that $\FF$ is \emph{transversely affine} if it is a transversely homogeneus $(\Aff^+(\R),\R)$-foliation\footnote{We only consider transversely affine foliations that are transversely oriented. The group $\Aff(\R)$ of affine transformations would define transversely affine foliations that may not be transversely oriented.}. This means that, according to Section~\ref{ss: prelim,  folns of codim 1}, $\omega$ and $\theta$ can be chosen so that $d\theta=0$ \cite{Seke1980}; it will be said that the transversely affine foliation $\FF$ is defined by $(\omega,\theta)$. In this case, the description of Section~\ref{ss: homogeneous folns} is given by $\pi:\widetilde M\to M$, $D:\widetilde M\to\R$, $h:\Gamma\to\Aff^+(\R)$ and $\Hol\FF\subset\Aff^+(\R)$. 

Assume that $\FF$ is transversely affine. Then $\Gamma\ne\{e\}$ because $D(\widetilde M)$ is open in $\R$. Furthermore $\FF$ has a finite number of compact leaves with holonomy \cite[Proposition~III.3.10]{Godbillon1991}, but non-compact leaves may also have holonomy. A theorem of Inaba \cite[Theorem~1.2]{Inaba1989} states that, either $\FF$ is almost without holonomy and $\Hol\FF$ is abelian (the \emph{elementary} case), or $\FF$ has a locally dense resilient leaf and $\Hol\FF$ is non-abelian.

From now on, consider only the elementary case. Then:
\begin{enumerate}[(a)]

\item\label{i: a group of translations} either $\Hol\FF$ is a group of translations; or 

\item\label{i: a group of homotheties} $\Hol\FF$ is conjugate by some translation to a group of homotheties. 

\end{enumerate}

In the case~(\ref{i: a group of translations}), $\FF$ is an $\R$-Lie foliation on a closed manifold, whose Fedida's description is given by $\pi$, $D$ and $h$; in particular, $\im D=\R$. 

In the case~(\ref{i: a group of homotheties}), after conjugation, we can assume that $\Hol\FF$ is indeed a group of homotheties. Since $\im D$ is $\Hol\FF$-invariant and $\Hol\FF\ne\{\id_\R\}$, either $\im D=\R^\pm$, or $\im D=\R$. If $\im D=\R^\pm$, we can pass to a group of translations by using $\ln|D|$ instead of $D$. Thus, if $\FF$ is not an $\R$-Lie foliation, we can assume that $\Hol\FF$ is a non-trivial group of homotheties and $\im D=\R$. Let us analyze this case using the notation of Section~\ref{ss: folns almost w/o hol}.

\begin{lem}\label{l: M^0 = pi(D^{-1}(0)), ...}
\begin{enumerate}[(i)]

\item\label{i: M^0 = pi(D^{-1}(0))} $M^0=\pi(D^{-1}(0))$.

\item\label{i: non-trivial subgroups of Gamma_0} The holonomy groups of leaves in $M^0$ are isomorphic to non-trivial subgroups of $\Hol_0\FF$.

\item\label{i: the same model} All foliations $\FF_l$ have the same model, either~(\ref{i: model 1}) with $\partial M_l\ne\emptyset$, or~(\ref{i: model 2}).

\end{enumerate}
\end{lem}

\begin{proof}
By Proposition~\ref{p: quasi-analytic => the same models}, all foliations $\FF_l$ have the same model, which is neither~(\ref{i: model 0}), nor~(\ref{i: model 1}) with $\partial M_l=\emptyset$, otherwise $\FF$ would be an $\R$-Lie foliation. Thus~(\ref{i: the same model}) holds. It also follows that the holonomy groups of the leaves in $M^0$ cannot be trivial, obtaining ``$\subset$'' in~(\ref{i: M^0 = pi(D^{-1}(0))}) because $\Hol_0\FF$ is the only non-trivial isotropy group. Hence~(\ref{i: non-trivial subgroups of Gamma_0}) is true by~\eqref{Hol L < Hol_x FF}. 

There is a regular foliated atlas $\{U_k,x_k\}$ of $\FF$ such that, for every $k$, there is foliated chart $(\widetilde U_k,\tilde x_k)$ of $\widetilde\FF$ so that $\pi:\widetilde U_k\to U_k$ is a diffeomorphism, $\tilde x_k=x_k\pi$ and $\tilde x'_k=D|_{\widetilde U_k}$. Hence $D^{-1}(0)$ contains just one plaque of every $(\widetilde U_k,\tilde x_k)$. Since $\{U_k,x_k\}$ is finite, and $D^{-1}(0)$ is $\Gamma$-invariant because $0$ is fixed by $\Hol\FF$, it follows that $\pi(D^{-1}(0))$ contains a finite number of plaques of the  foliated atlas $\{U_k,x_k\}$. So $\pi(D^{-1}(0))$ is a finite union of compact leaves because $\{U_k,x_k\}$ is regular. This shows ``$\supset$'' in~(\ref{i: M^0 = pi(D^{-1}(0))}) by~(\ref{i: the same model}) and Remark~\ref{r: Hector's description of folns almost w/o hol}~(\ref{i: the leaves of mathring FF_l approach any leaf in the boundary}).
\end{proof}

Note that $x\partial_x\in\fX(\R)$ is invariant by homotheties. Let $\Diffeo(\R,0)\subset\Diffeo(\R)$ denote the subgroup of diffeomorphisms that fix $0$.

\begin{lem}\label{l: Z = varkappa x partial_x}
\begin{enumerate}[(i)]

\item\label{i: Z = varkappa x partial_x} If $Z\in\fX(\R)$ is invariant by some homothety $h\ne\id_\R$, then $Z=\varkappa x\partial_x$ for some $\varkappa\in\R$.

\item\label{i: h is a homothety} If $h\in\Diffeo(\R,0)$ preserves $x\partial_x$, then $h$ is a homothety.

\end{enumerate}
\end{lem}

\begin{proof}
Let us prove~(\ref{i: Z = varkappa x partial_x}). We can assume $h(x)=\lambda x$ ($x\in\R$) for some $\lambda>1$; otherwise consider $h^{-1}$. Any $h$-invariant $Z\in\fX(\R)$ vanishes at $0$ because this is the only fixed point of $h$. Thus $Z=xf \partial_x$ for some $f\in C^\infty(\R)$. From the $h$-invariance of both $Z$ and $x\partial_x$, and since $x\partial_x$ only vanishes at $x=0$, we get that $f$ is $h$-invariant. So $f(0)=\lim_{m\to\infty}f(x/\lambda^m)=f(x)$ for all $x\in\R$; i.e., $f$ is constant.

Let us prove~(\ref{i: h is a homothety}). Since $h$ preserves $x\partial_x$, it commutes with the flow of $x\partial_x$; i.e., $h(e^tx)=e^th(x)$ for all $x,t\in\R$. Therefore $x\mapsto h(x)/x$ is constant on $\R^\pm$. Since $h$ is smooth at zero, it follows that $h$ is a homothety.
\end{proof}

\begin{rem}\label{r: Z = varkappa x partial_x}
The same arguments can be used to show versions of Lemma~\ref{l: Z = varkappa x partial_x} on intervals $J$ of the form $(-\epsilon,\epsilon)$, $[0,\epsilon)$ or $(-\epsilon,0]$ ($0<\epsilon\le\infty$):
\begin{enumerate}[(i)]

\item\label{i: Z = varkappa x partial_x, on J} If $Z\in\fX(J)$ is invariant by the restriction to $J$ of the pseudogroup generated by some homothety $\ne\id_\R$, then $Z=\varkappa x\partial_x$ for some $\varkappa\in\R$.

\item\label{i: h is a homothety, on J} If a smooth pointed embedding $h:(J,0)\to(\R,0)$ preserves $x\partial_x$, then $h$ is the restriction of a homothety.

\end{enumerate}
\end{rem}

By Lemma~\ref{l: Z = varkappa x partial_x}~(\ref{i: Z = varkappa x partial_x}), $\fX(\R,\Hol\FF)=\R\,x\partial_x$. Let $\overline Z\in\olfX(M,\FF)$ be defined by $x\partial_x\in\fX(\R,\Hol\FF)$ according to~\eqref{olfX(M, FF) supset fX(im D, Hol FF)}. By Lemma~\ref{l: M^0 = pi(D^{-1}(0)), ...}~(\ref{i: M^0 = pi(D^{-1}(0))}), the zero set of $\overline Z$ is $M^0$. Thus $\FF^1_l\equiv\mathring\FF_l$ becomes a complete $\R$-Lie foliation with the restriction of $\overline Z$ to every $M^1_l\equiv\mathring M_l$, without having to change the differentiable structure (cf.\ Remark~\ref{r: Hector's description of folns almost w/o hol}~(\ref{i: models (2) are Lie foliations})).

\begin{lem}\label{l: overline Z is determined by overline Z|_V}
For any neighborhood $V$ in $M$ of a leaf $L\subset M^0$, every $\overline Z\in\overline\fX(M,\FF)$ is determined by $\overline Z|_V$.
\end{lem}

\begin{proof}
With the notation of Remark~\ref{r: Hector's description of folns almost w/o hol}~(\ref{i: Hol_pm L}) for this particular $L$, any leaf of $\FF^1_{l_\pm}$ meets $V$ by Remark~\ref{r: Hector's description of folns almost w/o hol}~(\ref{i: the leaves of mathring FF_l approach any leaf in the boundary}). So the restriction $\overline Z$ to $\overline{M^1_{l_+}\cup M^1_{l_-}}$ is determined by $\overline Z|_V$. By Lemma~\ref{l: Z = varkappa x partial_x}~(\ref{i: Z = varkappa x partial_x}) and Remark~\ref{r: Z = varkappa x partial_x}~(\ref{i: Z = varkappa x partial_x, on J}), and using the Reeb's local stability, it follows that the restriction $\overline Z$ to some neighborhood of $\overline{M^1_{l_+}\cup M^1_{l_-}}$ is also determined by $\overline Z|_V$. Then we can apply the same argument to all closures $\overline{M^1_l}$ that meet $\overline{M^1_{l_+}\cup M^1_{l_-}}$. Continuing in this way, the result follows because $M$ is connected.
\end{proof}

\begin{prop}\label{p: olfX(M, FF) equiv fX(im D, Gamma), transversely affine}
$\olfX(M,\FF)\equiv\fX(\R,\Hol\FF)$ via~\eqref{olfX(M, FF) supset fX(im D, Hol FF)}.
\end{prop}

\begin{proof}
We have to prove that the injection of~\eqref{olfX(M, FF) supset fX(im D, Hol FF)} is surjective in this case. Let $Z\in\fX(\widetilde M/\widetilde\FF,\Gamma)$. Take leaves, $L\subset M^0$ of $\FF$ and $\widetilde L$ of $\widetilde\FF$ with $\pi(\widetilde L)=L$. There are open neighborhoods, $V$ of $\widetilde L$ in $\widetilde M/\widetilde\FF$ and $W$ of $0$ in $\R$, so that $\overline D:V\to W$ is a diffeomorphism. Consider $\{e\}\ne\Hol L\subset\Hol_0\FF$ according to~\eqref{Hol L < Hol_x FF}. By Lemma~\ref{l: Z = varkappa x partial_x}~(\ref{i: Z = varkappa x partial_x}) and Remark~\ref{r: Z = varkappa x partial_x}~(\ref{i: Z = varkappa x partial_x, on J}), $D_*(Z|_V)=\varkappa x\partial_x|_V$ for some $\varkappa\in\R$ if $V$ and $W$ are small enough. So, by Lemma~\ref{l: overline Z is determined by overline Z|_V}, $Z$ corresponds to $\varkappa x\partial_x\in\fX(\R,\Hol\FF)$ via~\eqref{olfX(M, FF) supset fX(im D, Hol FF)}.
\end{proof}

The transverse orientation of every $\FF_l$ is directed, either outward on all boundary leaves of $M_l$, or inward on all of them \cite[Lemma~3.4]{Inaba1989}. Thus no pair of boundary components of the same $M_l$ is glued to get $M$. So, not only $\mathring M_l\equiv M^1_l$, but also $M_l\equiv\overline{M^1_l}$ via $\bfpi$. In particular, there have to be at least two manifolds $M_l$, and $M^0$ contains at least two leaves.

\begin{ex}\label{ex: trasv aff foln on S^n-1 times S^1}
Let $\widetilde\FF$ denote the foliation on $\widetilde M:=\R^n\sm\{0\}$ ($n>1$) whose leaves are the connected components of the last coordinate projection $D:\widetilde M\to\R$. Multiplication by any $\lambda>1$ defines an action of $\Z$ on $\widetilde M$, giving rise to a covering $\pi_\lambda:\widetilde M\to M_\lambda$, where $M_\lambda$ is diffeomorphic to $S^{n-1}\times S^1$. Since $\widetilde\FF$ is $\Z$-invariant, it induces an elementary transversely affine foliation $\FF_\lambda$ on $M_\lambda$, being $\pi_\lambda$ and $D$ the maps of its description of Section~\ref{ss: homogeneous folns}. $M_\lambda^0=\pi_\lambda(D^{-1}(0))$ is diffeomorphic to $S^{n-2}\times S^1$. Thus there are two compact leaves if $n=2$, and one compact leaf if $n>2$. $M_\lambda^1$ has two components, $M^1_{\lambda,\pm}=\pi_\lambda(\R^\pm)$. The corresponding foliated manifolds with boundary, $(M_{\lambda,\pm},\FF_{\lambda,\pm})$, are transversely affine Reeb components on $D^{n-1}\times S^1$ \cite[Section~1.4.4]{HectorHirsch1981-A}, using the obvious extension of this property to foliations on manifolds with boundary. A different description of these transversely affine Reeb components is given in \cite[Example~1.1.12]{CandelConlon2000-I}. 
\end{ex}

\begin{ex}\label{ex: connected sum of transv affine folns}
Consider the standard affine structure on $\R$, and its restriction to $\R^+$. The affine circles are \cite{Kuiper1953a}, \cite[Appendix to Section~2]{Goldman1980}: 
\begin{enumerate}[(i)]

\item\label{i: standard affine struct on S^1} the quotient of $\R$ by the additive action of $\Z$; and, 

\item\label{i: lambda affine struct on S^1} for every $\lambda>1$, the quotient of $\R^+$ by the multiplicative action of $\lambda\Z$. 

\end{enumerate}
After fixing an orientation, affine structures on $S^1$ are the transversely affine structures $(\omega,\theta)$ of the foliation by points. Then the affine structure defined by $(\omega,\theta)$ is isomorphic to~(\ref{i: standard affine struct on S^1}) if $\int_{S^1}\theta=0$, and isomorphic to~(\ref{i: lambda affine struct on S^1}) for some $\lambda>1$ if $|\int_{S^1}\theta|=\ln\lambda$. Thus $|\int_{S^1}\theta|$ classifies these structures on $S^1$; indeed, $\int_{S^1}\theta$ classifies these structures up to orientation preserving isomorphisms \cite[Section~III.3.3]{Godbillon1991}, \cite[Section~4.1]{Seke1980}.

Now let $\FF$ be a transversely affine foliation on a closed manifold $M$ defined by $(\omega,\theta)$. Any smooth closed transversal $c:S^1\to M$ of $\FF$ induces the orientation and affine structure on $S^1$ given by $(c^*\omega,c^*\theta)$.

In Example~\ref{ex: connected sum of folns}, suppose $\FF$ and $\GG$ are transversely affine, defined by $(\omega,\theta)$ and $(\alpha,\beta)$, respectively. If they induce the same orientation and affine structure on $S^1$ via $c$ and $d$ ($c^*\omega=f\,d^*\alpha$ for some $0<f\in C^\infty(S^1)$ and $\int_{S^1}c^*\theta=\int_{S^1}c^*\beta$), then $\FF'$ clearly becomes transversely affine.

In Example~\ref{ex: trasv aff foln on S^n-1 times S^1}, let $c_{\lambda,\pm}:S^1\to M_\lambda$ be a smooth closed transversal of $\FF_\lambda$ that cuts every leaf of $\FF^1_{\lambda,\pm}$ once, and induces the standard orientation of $S^1$. Via $c_{\lambda,\pm}$, we get the affine structure~(\ref{i: lambda affine struct on S^1}) on $S^1$ defined with $\lambda$. 

In Example~\ref{ex: turbulization}, if $\FF$ is also transversely affine, inducing the standard orientation on $S^1$ via $c$, then there is a transversely affine turbulization along $c$ if and only if $\ln\lambda:=\int_{S^1}c^*\theta\ne0$ (taking the connected sum with $\FF_\lambda$ along $c$ and $c_{\lambda,\pm}$) \cite[Section~2]{Seke1980}.
\end{ex}

\subsection{Transversely projective foliations}\label{ss: transv proj folns}

Recall that $\SL(2,\R)$ is the Lie group of $2\times2$ matrices of determinant one, and $\PSL(2,\R)=\SL(2,\R)/\{\pm I\}$, where $I$ denotes the identity matrix. $\PSL(2,\R)$ acts on the projective line $S^1_\infty=\R\cup\{\infty\}$ by projective transformations, the action of $\left(\begin{smallmatrix}a&b\\c&d\end{smallmatrix}\right)$ being $x\mapsto(ax+b)/(cx+d)$. The stabilizer of $\infty$ consists of the upper triangular matrices ($c=0$), whose restriction to $\R$ gives $\Aff^+(\R)$. An element $A\in\PSL(2,\R)$ is called \emph{hyperbolic}, \emph{parabolic} or \emph{elliptic} if it has $2$, $1$ or $0$ fixed points in $S^1_\infty$, respectively. Elliptic elements are conjugate to rotations (elements of $\PSO(2)=\SO(2)/\{\pm I\}$) different from the identity. The hyperbolic and parabolic elements are conjugate to transformations of the form $x\mapsto\lambda x$ ($\lambda>0$) and $x\mapsto x+\lambda$ ($\lambda\ne0$), respectively.

It is said that $\FF$ is \emph{transversely projective} if it is a transversely homogeneus $(\PSL(2,\R),S^1_\infty)$-foliation. This means that, according to Section~\ref{ss: prelim,  folns of codim 1}, $\omega$ and $\theta$ can be chosen so that $d\theta=\eta\wedge\omega$ and $d\eta=\eta\wedge\theta$ for some $\eta\in C^\infty(M;\Lambda^1)$ \cite{Blumenthal1979}. In this case, the corresponding description of Section~\ref{ss: homogeneous folns} is given by $\pi:\widetilde M\to M$, $D:\widetilde M\to S^1_\infty$, $h:\Gamma\to\PSL(2,\R)$ and $\Hol\FF\subset\PSL(2,\R)$. 

Assume that $\FF$ is transversely projective and almost without holonomy. Then Inaba and Matsumoto proved that either of the following holds \cite[Proposition~2.1, the proof of Proposition~3.4 and its remark]{InabaMatsumoto1990}:
\begin{enumerate}[(a)]

\item\label{i: abelian subgroup of PSO(2)} $\Hol\FF$ is conjugate to an abelian subgroup of $\PSO(2)$. 

\item\label{i: hyperbolic elements with a common fixed point set} $\Hol\FF$ consists of the identity, hyperbolic elements with a common fixed point set and possible elliptic elements which keep the fixed point set invariant.

\item\label{i: subgroup of the stabilizer at infty} $\Hol\FF$ is conjugate to a subgroup of the stabilizer of $\infty$.

\end{enumerate}

In the case~(\ref{i: abelian subgroup of PSO(2)}), $\FF$ is an $\R$-Lie foliation. 

In the case~(\ref{i: subgroup of the stabilizer at infty}), we can assume that $\Hol\FF$ is a subgroup of the stabilizer of $\infty$ after conjugation. If $\infty\notin\im D$, then $\FF$ is transversely affine. If $\infty\in\im D$ and $\Hol\FF$ does not contain parabolic elements, then $\FF$ satisfies~(\ref{i: hyperbolic elements with a common fixed point set}). If $\infty\in\im D$ and $\Hol\FF$ has some parabolic element $h$, then the fixed point of $h$ is $\infty$, and $\pi(D^{-1}(\infty))$ consists of some compact leaves whose holonomy group cannot be given by germs of homotheties.

In the case~(\ref{i: hyperbolic elements with a common fixed point set}), $\Hol\FF$ is virtually abelian, and it is abelian just when there are no elliptic elements. After conjugation, we can assume that the fixed point set of the hyperbolic elements is $\{0,\infty\}$. Since $\im D$ is $\Hol\FF$-invariant and $\widetilde M$ is connected, it follows that $\im D$ is $\R^\pm$, $\R$, $S^1_\infty\sm\{0\}$ or $S^1_\infty$. If $\im D=\R^\pm$ or $\im D=\R$, then $\FF$ is transversely affine. If $\im D=S^1_\infty\sm\{0\}$, then we pass to the case $\im D=\R$ using conjugation by the rotation $x\mapsto-1/x$ of $S^1_\infty$. Thus, if $\FF$ is not transversely affine, then $\im D=S^1_\infty$. Let us analyze the last case from now on.

Now an obvious version of Lemma~\ref{l: M^0 = pi(D^{-1}(0)), ...} follows with a similar proof, where $D^{-1}(\{0,\infty\})$ is used in~(\ref{i: M^0 = pi(D^{-1}(0))}) instead of $D^{-1}(0)$, and subgroups of $\Hol_0\FF$ or $\Hol_\infty\FF$ are used in~(\ref{i: non-trivial subgroups of Gamma_0}) instead of just subgroups of $\Hol_0\FF$.

Note that $x\partial_x\in\fX(\R)$ extends to a smooth vector field on $S^1_\infty$, also denoted by $x\partial_x$, which is invariant by all hyperbolic elements with fixed point set $\{0,\infty\}$. In fact, $x\partial_x$ on $S^1_\infty\sm\{0\}$ corresponds to $-y\partial_y$ on $\R$ by the rotation $x\mapsto y=-1/x$ of $S^1_\infty$.

\begin{lem}\label{l: Z = varkappa x partial_x, transv proj}
If $Z\in\fX(S^1_\infty)$ is invariant by some hyperbolic element whose fixed point set is $\{0,\infty\}$, then $Z=\varkappa x\partial_x$ for some $\varkappa\in\R$. In particular, $\fX(S^1_\infty,\Hol\FF)=\R\,x\partial_x$ if $\Hol\FF$ has no elliptic element, otherwise $\fX(S^1_\infty,\Hol\FF)=0$.
\end{lem}

\begin{proof}
By Lemma~\ref{l: Z = varkappa x partial_x}~(\ref{i: Z = varkappa x partial_x}), $Z|_\R=\varkappa x\partial_x$ for some $\varkappa\in\R$ because the restriction to $\R$ of any hyperbolic element with fixed point set $\{0,\infty\}$ is a homothety different from the identity. So $Z=\varkappa x\partial_x$ on $S^1_\infty$.

The last assertion is true because any elliptic element $A$ preserving $\{0,\infty\}$ is conjugated to the rotation $x\mapsto-1/x$ by some hyperbolic element with fixed point set $\{0,\infty\}$, and therefore $A_*(x\partial_x)=-x\partial_x$.
\end{proof}

Like in Section~\ref{ss: transv affine folns}, every $\FF^1_l\equiv\mathring\FF_l$ becomes a complete $\R$-Lie foliation with the restriction to $M^1_l\equiv\mathring M_l$ of the element of $\olfX(M,\FF)$ defined by $x\partial_x\in\fX(\R,\Hol\FF)$ via~\eqref{olfX(M, FF) supset fX(im D, Hol FF)}. Moreover the statements of Lemma~\ref{l: overline Z is determined by overline Z|_V} and Proposition~\ref{p: olfX(M, FF) equiv fX(im D, Gamma), transversely affine} hold as well, with the obvious adaptations of the proofs.

Now the transverse orientation of every $\FF_l$ may be directed outward and inward on different boundary leaves of $M_l$. Anyway, $M^0$ contains at least two leaves because $\emptyset\ne\pi(D^{-1}(0)),\pi(D^{-1}(\infty))\subset M^0$.

\begin{ex}\label{ex: transv proj suspension foln}
The identity and the hyperbolic elements with common fixed point set $\{0,\infty\}$ form an abelian and torsion free subgroup $H\subset\PSL(2,\R)$ (its restriction to $\R$ is the group of orientation preserving homotheties). Let $\Gamma\subset H$ be a subgroup of finite rank, and let $\widetilde L$ be a $\Gamma$-covering of the closed oriented surface $L$ of genus two. Let $\widetilde M=S^1_\infty\times\widetilde L$ with the foliation $\widetilde\FF$ by the fibers of the first factor projection $D:\widetilde M\to S^1_\infty$. The diagonal action of $\Gamma$ on $\widetilde M$, given by $\gamma\cdot(x,\tilde y)=(\gamma(x),\gamma\cdot\tilde y)$, preserves $\widetilde\FF$. Thus it induces a suspension foliation $\FF$ on the closed manifold $M=\Gamma\backslash\widetilde M$ \cite[Section~3.1]{CandelConlon2000-I}. $\FF$ is a transversely projective foliation, whose developing map is $D$ and with $\Hol\FF=\Gamma$ (Section~\ref{ss: homogeneous folns}). It has two compact leaves, which are diffeomorphic to $L$, and all other leaves are diffeomorphic to $\widetilde L$.
\end{ex}

\begin{ex}\label{ex: transv proj model, n=2}
In Example~\ref{ex: Reeb components}~(\ref{i: |f(x)| = ln(1-|x|)^mu}), the model~(\ref{i: model 1}) foliation $\GG$ is transversely projective. It is transversely affine if and only if $\sign(f(x))$ has the same limit as $x\to1$ and as $x\to-1$, which is another description of the transversely affine Reeb component of Example~\ref{ex: trasv aff foln on S^n-1 times S^1} for $n=2$ and $\lambda=e^{1/\mu}$.

In Example~\ref{ex: tangential gluing of models}~(\ref{i: foln on T^2 or K}), using the above model~(\ref{i: model 1}) foliations to make tangential gluing, all of them with the same $\mu$, the result is a transversely projective foliation if it is transversely oriented, which means that the number of transversely affine models is even. It is transversely affine if and only if all models are transversely affine.
\end{ex}

\begin{ex}\label{ex: connected sum of transv proj folns}
In Example~\ref{ex: connected sum of folns}, if $\FF$ and $\GG$ are also transversely projective, and induce the same projective structure on $S^1$ via $c$ and $d$, then $\FF'$ clearly becomes transversely projective. (See \cite[Appendix to Section~2]{Goldman1980} for the classification of projective circles.)
\end{ex}

\section{Transversely simple foliated flows}\label{s: transv simple foliated flows}

Let $\FF$ be a smooth foliation of codimension one on a manifold $M$. For the sake of simplicity, assume that $\FF$ is transversely oriented. Let $Z\in\fXcom(M,\FF)$ with foliated flow $\phi$. Let $M^0$ be the union of leaves preserved by $\phi$. The $\phi$-invariant set $M^0$ is closed in $M$ because it is the zero set of $\overline{Z}\in\olfX(M,\FF)\subset C^\infty(M;N\FF)$. Moreover $\phi$ is transverse to the leaves on the open set $M^1:=M\sm M^0$. So there is a canonical isomorphism $N\phi\cong T\FF$ on $M^1$, and $\FF$ is TC at every point of $M^1$ (Section~\ref{ss: Riem folns}); in particular, the leaves in $M^1$ have no holonomy. With the notation of Sections~\ref{ss: holonomy}--\ref{ss: fol maps}, let $\bar\phi$ be the $\HH$-equivariant local flow on $\Sigma$ generated by $\overline Z\in\fX(\Sigma,\HH)$. Via the homeomorphism $M/\FF\to\Sigma/\HH$ defined by the maps $x'_k$, the leaves preserved by $\phi$ correspond to the $\HH$-orbits preserved by $\bar\phi$, whose union is $\Fix(\bar\phi)$ because they are totally disconnected. 

\begin{defn}\label{d: transversely simple}
The leaves preserved by $\phi$ that correspond to simple fixed points of $\bar\phi$ are called \emph{transversely simple}. If all leaves preserved by $\phi$ are transversely simple, then $\phi$ {\rm(}or $Z${\rm)} is called \emph{transversely simple}.
\end{defn} 

Since $\dim\Sigma=1$, for all simple $\bar p\in\Fix(\bar\phi)$, there is some $\varkappa=\varkappa_{\bar p}\in\R^\times$ such that $\bar\phi^t_*\equiv e^{\varkappa t}$ on $T_{\bar p}\Sigma\equiv\R$. By the $\HH$-equivariance of $\bar\phi$, we easily get $\varkappa_{\bar p}=\varkappa_{\bar q}$ for all $\bar q\in\HH(\bar p)\subset\Fix(\bar\phi)$. Thus we can use the notation $\varkappa_L=\varkappa_{\bar p}$ if $\HH(\bar p)$ corresponds to the simple preserved leaf $L$.

\begin{lem}\label{l: overline Z = varkappa_Lx partial_x}
Let $\psi$ be a local flow on $\R$ with infinitesimal generator $X\in\fX(\R)$. If $0$ is a simple fixed point of $\psi$ with $\varkappa_0=\varkappa$, then there is a coordinate $x$ around $0$ in $\R$ so that $x(0)=0$ and $X=\varkappa x\partial_x$, and therefore $\psi^t(x)=e^{\varkappa t}x$.
\end{lem}

\begin{proof}
Let $u$ denote the standard coordinate of $\R$. The condition on $0$ means that  $X=f(u)\partial_u$ for some $f\in C^\infty(\R)$ with $f(0)=0$ and $f'(0)=\varkappa$. Then $f(u)=uh(u)$ for some $h\in C^\infty(\R)$ with $h(0)=\varkappa$. Hence there is some $g\in C^\infty(\R)$ such that $\varkappa-h(u)=ug(u)$. We look for some smooth function $x=x(u)$ around $0$ so that $x(0)=0$, $x'(0)\ne0$ and $\varkappa x\partial_x=X$. Thus $x(u)=ue(u)$ for some smooth function $e=e(u)$ defined around $0$ with $e(0)\ne0$. Since $\partial_u=x'(u)\partial_x$, we need $\varkappa ue(u)=uh(u)(e(u)+ue'(u))$ around $0$; i.e., $e'(u)/e(u)=(\varkappa-h(u))/uh(u)=g(u)/h(u)$. Any $e(u)=C\exp(\int_0^ug(v)/h(v)\,dv)$ with $C\ne0$ will do the job.
\end{proof}

\begin{rem}\label{r: transversely simple}
\begin{enumerate}[(i)]


\item\label{i: depends on overline Z} Since $\bar\phi$ and $\overline Z\in\fX(\Sigma,\HH)\equiv\overline\fX(M,\FF)$ determine each other, the condition on the preserved leaves of $\phi$ to be transversely simple depends only on $\overline Z\in\overline\fX(M,\FF)$.

\item\label{i: overline Z = varkappa_Lx partial_x}  By Lemma~\ref{l: overline Z = varkappa_Lx partial_x}, around any point $p$ in a transversely simple leaf $L\subset M^0$, there are foliated coordinates $(x,y)$ with $x(p)=0$ and $\overline Z=\varkappa_Lx\partial_x$.

\item If $\phi$ is transversely simple, then every closed orbit is contained in either $M^0$ or $M^1$, and all fixed points belong to $M^0$.

\end{enumerate}
\end{rem}

From now on, suppose that $\phi$ is transversely simple and $M$ is compact, unless otherwise stated. 

\begin{prop}\label{p: M^0 is a finite union of compact leaves}
 $M^0$ is a finite union of compact leaves.
\end{prop}

\begin{proof}
  Since $\Fix(\bar\phi)$ has no accumulation points in $\Sigma$ (Section~\ref{ss: simple flows}), every leaf $L$ in $M^0$ has a neighborhood $V$ such that $V\cap M^0=L$. Thus the result follows using that $M$ is compact, and $M^0$ is closed in $M$.
\end{proof}

By Proposition~\ref{p: M^0 is a finite union of compact leaves} and since the leaves in $M^1$ have no holonomy, $\FF$ is almost without holonomy (Section~\ref{ss: folns almost w/o hol}), and only a finite number of leaves may have holonomy. According to Remark~\ref{r: Hector's description of folns almost w/o hol}~(\ref{i: adding leaves to M^0}), we can consider Hector's description with this choice of $M^0$ and $M^1$, even though there may be leaves without holonomy in $M^0$. Consider also the rest of the notation of Section~\ref{ss: folns almost w/o hol}. If the leaves in $M^1$ are not compact, then $M^1$ is just the transitive point set of $\FF$.

Given any leaf $L\subset M^0$ and $p\in L$, let $(x,y):U\to\Sigma\times B''$ be a foliated chart around $p$ like in Remark~\ref{r: transversely simple}~(\ref{i: overline Z = varkappa_Lx partial_x}), where $\Sigma$ is some open interval containing $0$. Let $\bfh:\pi_1 L\to\Hol L$, $\sigma\mapsto\bfh_\sigma$, be the holonomy homomorphism of $L$ at $p$. Via the projection $x:U\to\Sigma$, we can regard $\Hol L$ as a subgroup of the group of germs at $0$ of local transformations of $\Sigma$ such that $0$ is a fixed point in their domains.

\begin{prop}\label{p: Hol L}
$\Hol L$ consists of germs at $0$ of homotheties on $\R$.
\end{prop}

\begin{proof}
All elements of $\Hol L$ can be represented by elements of the group $\Diffeo^+(\R,0)$ of orientation-preserving diffeomorphisms of $\R$ that fix $0$. According to Remark~\ref{r: transversely simple}~(\ref{i: overline Z = varkappa_Lx partial_x}), for the above foliated coordinates $(x,y)$ around $p$, we have $\overline{Z}=\varkappa x\partial_x$ for $\varkappa=\varkappa_L$. Then, by Lemma~\ref{l: Z = varkappa x partial_x}~(\ref{i: h is a homothety}) and Remark~\ref{r: Z = varkappa x partial_x}~(\ref{i: h is a homothety, on J}), any element of $\Hol L$ is the germ at $0$ of a homothety.
\end{proof}

According to Proposition~\ref{p: Hol L}, $\bfh=\bfh_L$ is induced by the homomorphism $\hat h=\hat h_L:\pi_1L\to\Diffeo^+(\R,0)$ whose image consists of homotheties. We get an induced monomorphism $h=h_L:\Gamma:=\pi_1 L/\ker\hat h\to\Diffeo^+(\R,0)$, $\gamma\mapsto h_\gamma$, with $h_\gamma(x)=a_\gamma x$ for some monomorphism $\Gamma\to\R^+\equiv(\R^+,\times)$, $\gamma\mapsto a_\gamma=a_{L,\gamma}$. The holonomy cover $\pi=\pi_L:\widetilde L\to L$ is determined by $\pi_1\widetilde L\equiv\ker\hat h=\ker\bfh$. On some neighborhood of $L$, $\FF$ can be described with the suspension defined by $\pi$ and $h$, recalled in Section~\ref{s: suspension}. 


Every $\FF^1_l$ becomes a complete $\R$-Lie foliation with the structure induced by $Z_l\in\fXcom(M^1_l,\FF^1_l)$, with the original differentiable structure (see Remark~\ref{r: Hector's description of folns almost w/o hol}~(\ref{i: models (2) are Lie foliations})). We use the following notation for its Fedida's description (Sections~\ref{ss: Riem folns} and~\ref{ss: R-Lie folns}): $\pi_l:\widetilde M^1_l\to M^1_l$, $h_l:\Gamma_l:=\Aut(\pi_l)\to\R$, $D_l:\widetilde M^1_l\to\R$ and $\widetilde\FF^1_l=\pi_l^*\FF^1_l$. The abelian and torsion free group $\Gamma_l$ has finite rank because $\pi_1M^1_l\equiv\pi_1\mathring M_l\cong\pi_1M_l$ and $M_l$ is compact. The action of any $\gamma\in\Gamma_l$ on $\widetilde M^1_l$ is denoted by $\tilde p\mapsto\gamma\cdot\tilde p$ or by $T_\gamma$. Let $\widetilde Z_l$ and $\tilde\phi_l$ be the lifts of $Z_l$ and $\phi_l$ to $\widetilde M^1_l$. Recall that $\widetilde Z_l$ is $D_l$-projectable, and we can assume that $D_{l*}\widetilde Z_l=\partial_x$ (Section~\ref{ss: R-Lie folns}).

By Remark~\ref{r: Hector's description of folns almost w/o hol} and Proposition~\ref{p: quasi-analytic => the same models}, we have the following cases for $\FF$:
\begin{enumerate}[(a)]

\item\label{i: fiber bundle} $\FF$ is given by a fiber bundle $M\to S^1$ with connected fibers.

\item\label{i: minimal Lie foln} $\FF$ is an $\R$-Lie foliation with dense leaves.

\item\label{i: all folns FF_l are models (1)} $M^0\ne\emptyset$, $\Hol L\cong\Z$ for all leaves $L\subset M^0$, and the foliations $\FF^1_l$ are given by fiber bundles $M^1_l\to S^1$ with connected fibers.

\item\label{i: all folns FF_l are models (2)} $M^0\ne\emptyset$, $\Hol L$ is a finitely generated abelian group of rank $>1$ for all leaves $L\subset M^0$, and all foliations $\FF^1_l$ are minimal $\R$-Lie foliations.

\end{enumerate}
The case~(\ref{i: fiber bundle}) can be considered as a model~(\ref{i: model 1}) with empty boundary, avoiding the use of models~(\ref{i: model 0}), or it can be cut into models~(\ref{i: model 0}) by adding a finite number of leaves without holonomy to $M^0$ (Remark~\ref{r: Hector's description of folns almost w/o hol}~(\ref{i: adding leaves to M^0})).

\begin{rem}\label{r: M is not compact, transv simple foliated flows}
The results and observations of this section hold without requiring $M$ to be compact, assuming only that $M^0$ is compact.
\end{rem}

\begin{ex}\label{ex: no transv simple foliated flows}
By Proposition~\ref{p: Hol L}, the Reeb foliation $\FF$ on $S^3$ does not admit any transversely simple foliated flow because it has a leaf with holonomy but no infinitesimal holonomy. Actually, this proves that its Reeb components of \cite[Example~3.3.11]{CandelConlon2000-I} cannot show up as models in Hector's description of any foliation on a closed manifold with a simple foliated flow. Similarly, this realization is impossible for Example~\ref{ex: Reeb components}~(\ref{i: f(x) = exp 1/(1-x^2)}),(\ref{i: f(x) = x^2/(1-x^2)}).
\end{ex}

\section{Case of a suspension foliation}\label{s: suspension}

\subsection{Basic definitions}\label{ss: basic defns, suspension}

For a connected closed manifold $L$, let $\hat h:\pi_1L\to\Diffeo^+(\R,0)$ be a homomorphism whose image consists of homotheties (like in Section~\ref{s: transv simple foliated flows}). It induces a monomorphism $h:\Gamma:=\pi_1L/\ker\hat h\to\Diffeo^+(\R,0)$, $\gamma\mapsto h_\gamma$. We have $h_\gamma(x)=a_\gamma x$ for some monomorphism $\Gamma\to\R^+$, $\gamma\mapsto a_\gamma$; in particular, $\Gamma$ is abelian, torsion free and finitely generated. Let $\pi=\pi_L:(\widetilde L,\tilde p)\to(L,p)$ be the pointed regular covering map with $\pi_1\widetilde L=\pi_1(\widetilde L,\tilde p)\equiv\ker\hat h$, and therefore $\Aut(\pi)\equiv\Gamma$. We may use the notation $[\tilde y]=\pi(\tilde y)$ for $\tilde y\in\widetilde L$. The canonical left action of every $\gamma\in\Gamma$ on $\widetilde L$ is denoted by $T_\gamma$ or $\tilde y\mapsto\gamma\cdot\tilde y$. For the diagonal left action of $\Gamma$ on $\widetilde M=\R\times\widetilde L$, $\gamma\cdot(x,\tilde y)=(a_\gamma x,\gamma\cdot\tilde y)$, let $M=\Gamma\backslash\widetilde M$. The canonical projection $\pi_M:\widetilde M\to M$ is a $\Gamma$-cover with deck transformations $h_\gamma\times T_\gamma$ ($\gamma\in\Gamma$). Write $[x,\tilde y]=\pi_M(x,y)$ for $(x,\tilde y)\in\widetilde M$. Let $\widetilde\varpi:\widetilde M\to\widetilde L$ denote the second factor projection, and let $\widetilde\FF$ be the foliation on $\widetilde M$ with leaves $\{x\}\times\widetilde L$ ($x\in\R$). Since $\widetilde\varpi$ is $\Gamma$-equivariant, it induces a fiber bundle map $\varpi:M\to L$, defined by $\varpi([x,\tilde y])=[\tilde y]$. On the other hand, since $\widetilde\FF$ is $\Gamma$-invariant, it induces a foliation $\FF$ on $M$ so that $\pi^*\FF=\widetilde\FF$, which is transverse to the fibers of $\varpi$. $(M,\FF)$ is called the \emph{suspension} defined by $\hat h$ (or $h$) and $\pi$ \cite[Section~3.1]{CandelConlon2000-I}. Note that the typical fiber of $\varpi$ is $\R$ because the corresponding fibers of $\widetilde\varpi$ and $\varpi$ can be identified via $\pi_M$. Since $0$ is fixed by the $\Gamma$-action on $\R$, the leaf $\{0\}\times\widetilde L\equiv\widetilde L$ of $\widetilde\FF$ is $\Gamma$-invariant, and $\pi_M(\{0\}\times\widetilde L)\equiv L$ is a compact leaf of $\FF$. The other leaves of $\widetilde\FF$ are diffeomorphic via $\pi_M$ to the corresponding leaves of $\FF$ because the elements of $\Gamma\sm\{e\}$ have no fixed points in $\R^\times$. Given $\tilde y\in\widetilde L$ and $y=[\tilde y]\in L$, the fiber $\varpi^{-1}(y)\equiv\widetilde\varpi^{-1}(\tilde y)=\R\times\{\tilde y\}\equiv\R$ is a global transversal of $\FF$ through $[0,\tilde y]\equiv y$. Note that the holonomy homomorphism $\bfh:\pi_1L\to\Hol L$ is induced by $h$, and therefore $\widetilde L^{\text{\rm hol}}\equiv\widetilde L$. The standard orientation of $\R$ induces a transverse orientation of $\widetilde\FF$, which is $\Gamma$-invariant, giving rise to a transverse orientation of $\FF$. 

$\FF$ is transversely affine foliation on an open manifold. Its description of Section~\ref{ss: homogeneous folns} is given by $\pi_M:\widetilde M\to M$, the first factor projection $D:\widetilde M\to\R$ and $h:\Gamma\to\Aff^+(\R)$. In this case, $D$ induces an identity $\widetilde M/\widetilde\FF\equiv\R$, and therefore the inclusions of~(\ref{olfX(M, FF) supset fX(im D, Hol FF)}) and~(\ref{Hol L < Hol_x FF}) are equalities (cf.\ Proposition~\ref{p: olfX(M, FF) equiv fX(im D, Gamma), transversely affine} for the case where $M$ is closed).

\subsection{Transversely simple vector fields on a suspension foliation}\label{ss: vector fields, suspension}

Given any $\varkappa\in\R^\times$, consider the  transversely simple foliated flow $\tilde\xi$ on $(\widetilde M,\widetilde\FF)$ given by $\tilde\xi^t(x,\tilde y)=(e^{\varkappa t}x,\tilde y)$, whose infinitesimal generator is $\widetilde Y=(\varkappa x\partial_x,0)\in\fXcom(\widetilde M,\widetilde\FF)$.  With the notation of Section~\ref{s: transv simple foliated flows} for $\tilde\xi$, we have $\Fix(\tilde\xi^t)=\widetilde M^0=\{0\}\times\widetilde L\equiv\widetilde L$, and the orbits on $\widetilde M^1$ are the fibers of the restriction $\widetilde\varpi:\widetilde M^1\to\widetilde L$. Since $\tilde\xi^t$ is $\Gamma$-equivariant and $\widetilde Y$ is $\Gamma$-invariant, they can be projected to $M$ obtaining a transversely simple foliated flow $\xi^t$ with infinitesimal generator $Y\in\fX(M,\FF)$, satisfying $\Fix(\xi^t)=M^0=\pi_M(\widetilde M^0)\equiv L$, and the orbits on $M^1$ are the fibers of the restriction $\varpi:M^1\to\widetilde L$. Moreover $\overline Y\equiv\varkappa x\partial_x$ on $\R$ via~(\ref{olfX(M, FF) supset fX(im D, Hol FF)}), whose flow $\bar\xi$ is given by $\bar\xi^t(x)=e^{\varkappa t}x$.

$\FF^1_\pm\equiv\mathring\FF_\pm$ on $M^1_\pm\equiv\mathring M_\pm$ is a transversely complete $\R$-Lie foliation with the structure defined by $Y_\pm\in\fXcom(M^1_\pm,\FF^1_\pm)$ (see Remark~\ref{r: M is not compact, transv simple foliated flows}). In its Fedida's description (Section~\ref{ss: Riem folns}), $\widetilde M^1_\pm$ is the holonomy covering of $M^1_\pm$, whose group of deck transformations is also $\Gamma$. The developing map $D_\pm:\widetilde M^1_\pm\to\R$ and holonomy homomorphism $h_\pm:\Gamma\to\R$ can be chosen to be given by $D_\pm(x,y)=\varkappa^{-1}\ln|x|=:t$ and $h_\pm(\gamma)=\varkappa^{-1}\ln a_\gamma$, and therefore $\Hol\FF_\pm=\{\,\varkappa^{-1}\ln a_\gamma\mid\gamma\in\Gamma\,\}$. In this way, $(D_\pm)_*\widetilde Y_\pm=\partial_t$, like in Section~\ref{ss: R-Lie folns}.

Let $\phi$ be any transversely simple foliated flow on $M$, with infinitesimal generator $Z\in\fXcom(M,\FF)$, such that $M^0=L$. According to Remark~\ref{r: transversely simple}~(\ref{i: overline Z = varkappa_Lx partial_x}), we can assume $\bar\phi=\bar\xi$ and $\overline Z=\overline Y$. Then the lifts to $\widetilde M$, $\tilde\phi$ of $\phi$ and $\widetilde Z$ of $Z$, are of the form
\begin{equation}\label{tilde phi^t(tilde y,x), widetilde Z}
\tilde\phi^t(x,\tilde y)=(e^{\varkappa t}x,\tilde\phi_x^t(\tilde y))\;,\quad\widetilde Z=(\varkappa x\partial_x,\widetilde Z_x)\;,
\end{equation}
for smooth families, $\{\,\widetilde\phi^t_x\mid x,t\in\R\,\}\subset\Diffeo(\widetilde L)$ and $\{\,\widetilde Z_x\mid x\in\R\,\}\subset\fX(\widetilde L)$. In particular, $\widetilde Z_0$ is the restriction of $\widetilde Z$ to $\widetilde L\equiv\{0\}\times\widetilde L$, and its flow is $\phi_0=\{\tilde\phi_0^t\}$. Thus $\widetilde Z_0$ is $\Gamma$-invariant and $\tilde\phi_0$ is $\Gamma$-equivariant, inducing the restrictions of $Z$ and $\phi$ to $L$, denoted by $Z_0$ and $\phi_0$.

\begin{prop}\label{p: phi_0 simple => phi simple in M^0}
The flow $\phi_0$ is simple if and only if the fixed points and closed orbits of $\phi$ in $M^0$ are simple.
\end{prop}

\begin{proof}
Let $\tilde y\equiv(0,\tilde y)\in\widetilde L\equiv\widetilde M^0$ and $y=[\tilde y]\equiv[0,\tilde y]\in L\equiv M^0$. Suppose that $y\in\Fix(\phi_0)\equiv\Fix(\phi)\cap M^0$, and therefore $\tilde y\in\Fix(\tilde\phi_0)\equiv\Fix(\tilde\phi)\cap\widetilde M^0$. By~(\ref{tilde phi^t(tilde y,x), widetilde Z}),
\begin{gather*}
T_{[0,\tilde y]}M\equiv T_{(0,\tilde y)}\widetilde M\equiv\R\oplus T_{\tilde y}\widetilde L\equiv\R\oplus T_yL\;,\\
\phi^t_{*[0,\tilde y]}\equiv\tilde\phi^t_{*(0,\tilde y)}\equiv e^{\varkappa t}\oplus\tilde\phi^t_{0*\tilde y}
\equiv e^{\varkappa t}\oplus\phi^t_{0*y}\;.
\end{gather*}
So $p$ is simple for $\phi$ if and only if $y$ is simple for $\phi_0$. 

Now suppose that $y$ is in some closed orbit $c$ of $\phi_0$, which can be also considered as a closed orbit of $\phi$ in $M^0$. Then there is some $\gamma\in\Gamma$ such that $\tilde\phi_0^{\ell(c)}(\tilde y)=\gamma\cdot\tilde y$. As before,
\begin{gather*}
N_{[0,\tilde y]}\phi\equiv N_{(0,\tilde y)}\tilde\phi\equiv\R\oplus N_{\tilde y}\tilde\phi_0\equiv\R\oplus N_y\phi_0\;,\\
N_{[0,\tilde y]}\phi\equiv N_{(0,\gamma\cdot\tilde y)}\tilde\phi\equiv\R\oplus N_{\gamma\cdot\tilde y}\phi_0\equiv\R\oplus N_y\phi_0\;,\\
\phi^{\ell(c)}_{*[0,\tilde y]}\equiv\tilde\phi^{\ell(c)}_{*(0,\tilde y)}\equiv e^{\varkappa\ell(c)}\oplus\tilde\phi^{\ell(c)}_{0*\tilde y}
\equiv e^{\varkappa\ell(c)}\oplus\phi^{\ell(c)}_{0*y}\;.
\end{gather*}
So $c$ is simple for $\phi$ if and only if it is simple for $\phi_0$. 
\end{proof}

\begin{prop}\label{p: A -> B}
For every simple $A\in\fX(L)$ without closed orbits, there is some simple $B\in\fXcom(M,\FF)$ without closed orbits such that $\overline{B}=\overline Y$ and $B_0\equiv A$. 
\end{prop}

\begin{proof}
Let $\widetilde A\in\fXcom(\widetilde L)$ be the lift of $A$, whose flow is denoted by $\tilde\zeta$, and let $\widetilde B=(\varkappa x\partial_x,\widetilde A)\in\fX(\widetilde M,\widetilde\FF)$. Clearly, $\widetilde B_0\equiv\widetilde A$ and $\overline{\widetilde B}=\overline{\widetilde Y}$. Moreover $\widetilde B$ is complete because its flow $\tilde\eta$ is given by $\tilde\eta^t(x,\tilde y)=(e^{\varkappa t}x,\tilde\zeta^t(\tilde y))$. Since $\widetilde B$ is $\Gamma$-invariant, it induces some $B\in\fXcom(M,\FF)$ with flow $\eta$. 

\begin{claim}\label{cl: eta has neither fixed points nor closed orbits in M^1}
The flow $\eta$ has neither fixed points nor closed orbits in $M^1$.
\end{claim}

By absurdity, suppose that $\eta^t([x,\tilde y])=[x,\tilde y]$ for some $[x,\tilde y]\in M^1$ and $t>0$. Then there is some $\gamma\in\Gamma$ such that $\tilde\eta^t(x,\tilde y)=\gamma\cdot(x,\tilde y)$. Since $x\ne0$, this means that $e^{\varkappa t}=a_\gamma$ and $\tilde\zeta^t(\tilde y)=\gamma\cdot\tilde y$. Thus $\zeta^t(y)=y$ for $y=[\tilde y]$. Hence $y\in\Fix(\zeta)$ because $\zeta$ has no closed orbits, and therefore $\tilde y\in\Fix(\tilde\zeta)$. It follows that $\gamma\cdot\tilde y=\tilde y$, yielding $\gamma=e$. So $e^{\varkappa t}=1$, obtaining $\varkappa t=0$, a contradiction.

By Proposition~\ref{p: phi_0 simple => phi simple in M^0}, Claim~\ref{cl: eta has neither fixed points nor closed orbits in M^1} and since $\eta_0\equiv\zeta$, it follows that $\eta$ is simple without closed orbits.
\end{proof}

\subsection{Differential forms defining a suspension foliation}\label{ss: diff forms defining the suspension}

For $k=\rank\Gamma$, fix generators $\gamma_1,\dots,\gamma_k$ of $\Gamma$. Let $c_i$ be a piecewise smooth loop in $L$ based at $p$ such that $[c_i]\in\pi_1(L,p)$ defines $\gamma_i$, and let $a_i=a_{\gamma_i}$. By the universal coefficients and Hurewicz theorems, there are closed $1$-forms $\beta_1,\dots,\beta_k$ on $L$ so that $\delta_{ij}=\langle[\beta_i],[c_j]\rangle=\int_0^1c_j^*\beta_i$ and $\langle[\beta_i],\ker\hat h\rangle=0$. Thus every $\pi^*\beta_i$ is exact on $\widetilde L$. Let $\theta=-\ln(a_1)\,\beta_1-\dots-\ln(a_k)\,\beta_k$. Then $\tilde\theta=\pi^*\theta=dF$ for some $F\in C^\infty(\widetilde L)$. With some abuse of notation, let $\theta\equiv\varpi^*\theta$, $\tilde\theta\equiv\widetilde\varpi^*\tilde\theta$ and $F\equiv\widetilde\varpi^*F$. It is easy to check that $T_\gamma^*F=F-\ln a_\gamma$ on $\widetilde L$ for all $\gamma\in\Gamma$. Thus $\tilde\rho=e^Fx$ and $\tilde\omega=|\varkappa|^{-1}e^F\,dx$ are $\Gamma$-invariant on $\widetilde M$. Furthermore $\tilde\rho$ is a defining function of $\widetilde L$ on $\widetilde M$, $\tilde\omega$ defines $\widetilde\FF$, $d\tilde\omega=\tilde\theta\wedge\tilde\omega$ and $d\tilde\rho=\tilde\rho\tilde\theta+|\varkappa|\,\tilde\omega$. We get an induced defining function $\rho$ of $L$ on $M$, and an induced form $\omega$ defining of $\FF$, so that $d\omega=\theta\wedge\omega$ and $d\rho=\rho\theta+|\varkappa|\,\omega$. We also get $M\equiv\R_\rho\times L_\varpi$, giving rise to smaller tubular neighborhoods $T_\epsilon\equiv(-\epsilon,\epsilon)_\rho\times L_\varpi$ ($\epsilon>0$).

\subsection{Change of the differentiable structure}\label{ss: change of diff struct}

Given $0<\alpha\ne1$, let $f_\alpha:\R\to\R$ be the homeomorphism defined by $f_\alpha(x)=\sign(x)|x|^\alpha=x\,|x|^{\alpha-1}$. The restrictions $f_\alpha:\R^\pm\to\R^\pm$ are diffeomorphisms, but $f_\alpha$ is not diffeomorphism around $0$. Clearly, $f_\alpha(a_\gamma x)=a_\gamma^\alpha f_\alpha(x)$, and it is easy to check that $f_{\alpha*}(x\partial_x)=\alpha u\partial_u$ on $\R^\pm$, using the coordinate $u=f_\alpha(x)$.  Like in Section~\ref{ss: basic defns, suspension}, let $h_\alpha:\Gamma\to\Diffeo^+(\R,0)$ be the monomorphism defined by $h_{\alpha,\gamma}(u)=a_\gamma^\alpha u$, and let $(M_\alpha,\FF_\alpha)$ be the suspension defined with $h_\alpha$ and $\pi$. The foliated homeomorphism $\widetilde\Upsilon_\alpha=f_\alpha\times\id$ of $(\widetilde M,\widetilde\FF)$ is equivariant with respect to the $\Gamma$-actions defined by $h$ and $h_\alpha$, and therefore it induces a foliated homeomorphism $\Upsilon_\alpha:(M,\FF)\to(M_\alpha,\FF_\alpha)$. The restriction $\Upsilon_\alpha:(M^1,\FF^1)\to(M^1_\alpha,\FF^1_\alpha)$ is a diffeomorphism. 

A transversely simple foliated flow $\xi_\alpha$ on $(M_\alpha,\FF_\alpha)$, with infinitesimal generator $Y_\alpha$, can be defined like $\xi$ and $Y$ in Section~\ref{ss: vector fields, suspension}, using $\varkappa\alpha$ instead of $\varkappa$, and we get $\Upsilon_{\alpha*}Y=Y_\alpha$ on $M^1_\alpha$. With more generality, for any transversely simple foliated flow $\phi$ on $(M,\FF)$, with infinitesimal generator $Z\in\fXcom(M,\FF)$, such that $\bar\phi=\bar\xi$ and $\overline Z=\overline Y$, there is a transversely simple foliated flow $\phi_\alpha$ on $(M_\alpha,\FF_\alpha)$, with infinitesimal generator $Z_\alpha$, such that $\bar\phi_\alpha=\bar\xi_\alpha$, $\overline{Z_\alpha}=\overline{Y_\alpha}$, and $\Upsilon_{\alpha*}Z=Z_\alpha$ on $M^1_\alpha$. Precisely, using~(\ref{tilde phi^t(tilde y,x), widetilde Z}), their lifts $\tilde\phi_\alpha$ and $\widetilde Z_\alpha$ to $\widetilde M$ are given by
\[
\tilde\phi_\alpha^t(u,\tilde y)=(e^{\varkappa\alpha t}u,\tilde\phi_u^t(\tilde y))\;,\quad\widetilde Z_\alpha=(\varkappa\alpha u\partial_u,\widetilde Z_u)\;.
\] 

In other words, we get a new differentiable structure on $(M,\FF)$ via $\Upsilon_\alpha$, which agrees with the original one on $M^1$. This will be called a \emph{transverse power change} of the differentiable structure (around the leaf $L$). With this point of view, $\phi$ is a smooth transversely simple foliated flow with both differentiable structures, replacing $\varkappa$ with $\varkappa\alpha$. In this way, we can change $|\varkappa|$ arbitrarily, but keeping $\sign(\varkappa)$ invariant.

With the new differentiable structure, $C^\infty(M)$ is generated by $\rho_\alpha:=\rho\,|\rho|^{\alpha-1}$ and $C^\infty(L)\equiv\varpi^*C^\infty(L)$. Moreover $\rho_\alpha$ is a defining function of $L$, $\omega_\alpha:=\rho^{\alpha-1}\omega$ and $\theta_\alpha$ have smooth extensions to $M$, $\omega_\alpha$ defines $\FF$, $d\omega_\alpha=\theta_\alpha\wedge\omega_\alpha$ and $d\rho_\alpha=\rho_\alpha\theta_\alpha+|\alpha\varkappa|\,\omega_\alpha$.

\section{Global structure}\label{s: global str}

Consider the notation of Section~\ref{s: transv simple foliated flows}, where $M$ is compact, $\FF$ is transversely oriented, and $\phi$ is transversely simple. 

\subsection{Tubular neighborhoods of the components of $M^0$}\label{ss: tubular neighborhoods}

In the following, $L$ runs in $\pi_0M^0$ (the set of leaves in $M^0$), and we have corresponding objects $\hat h_L$, $h_L$, $\Gamma_L$, $\pi_L:\widetilde L\to L$, $a_{L,\gamma}$ and $\varkappa_L$, defined by $\FF$ and $\phi$. Consider the constructions of Sections~\ref{ss: basic defns, suspension}--\ref{ss: diff forms defining the suspension}, using this data, adding a prime and the subindex ``$L$'' to their notation: the suspension $(M'_L,\FF'_L)$ defined with $h_L$, with projection $\varpi'_L:M'_L\to L$, the transversely simple foliated flow $\xi'_L$ with infinitesimal generator $Y'_L$, the differential forms $\omega'_L$ and $\theta'_L$, the defining function $\rho'_L$, and the tubular neighborhoods $T'_{\epsilon,L}$.

By the Reeb's local stability, there are foliated diffeomorphisms between the restrictions of $\FF$ and $\FF'$ to tubular neighborhoods, $T_{L,0}$ of $L$ in $M$ and $T'_{L,0}:=T'_{L,\epsilon_0}$ ($\epsilon_0>0$) of $L$ in $M'_L$, so that the projection $\varpi_L$ of $T_{L,0}$ corresponds to the projection $\varpi'_L$ of $T'_{L,0}$. We will simply write $\FF\equiv\FF'_L$ and $\varpi_L\equiv\varpi'_L$ on $T_{L,0}\equiv T'_{L,0}$.  We can assume that the sets $\overline{T_{L,0}}$ are disjoint in $M$, and $\bar\phi\equiv\bar\xi'_L$ and $\overline Z\equiv\overline{Y'_L}$ on $T_{L,0}\equiv T'_{L,0}$ (Remark~\ref{r: transversely simple}~(\ref{i: overline Z = varkappa_Lx partial_x})). Fix also smaller tubular neighborhoods, $T_L\equiv T'_L:=T'_{L,\epsilon}$ ($0<\epsilon<\epsilon_0$).

Let $M'=\bigsqcup_LM'_L$, where we consider the combinations of all of the above objects, removing $L$ from the notation: $\FF'$, $\varpi'$, $\xi'$, $Y'$, $\omega'$, $\theta'$ and $\rho'$. Similarly, let $T'=\bigsqcup_LT'_L$, $T'_0=\bigsqcup_LT'_{L,0}$, $T=\bigcup_LT_L$ and $T_0=\bigcup_LT_{L,0}$. 

\begin{prop}\label{p: Z', A}
\begin{enumerate}[(i)]

\item\label{i: Z'} There is some $Z'\in\fXcom(M',\FF')$ such that $\overline{Z'}=\overline{Y'}$, $Z'\equiv Z$ on $T\equiv T'$, and $Z'=Y'$ on $M'\sm T'_0$.

\item\label{i: A -> B} For any $A\in\fX(M',\FF')$ with $\overline A=\overline{Y'}$, there is some $B\in\fX(M,\FF)$ with $\overline B=\overline Z$, $B\equiv A$ on $T\equiv T'$, and $B=Z$ on $M\sm T_0$.

\item\label{i: omega, theta, rho} There are $\omega,\theta\in C^\infty(M;\Lambda^1)$ such that $\omega$ defines $\FF$, $\omega\equiv\omega'$ on $T\equiv T'$ and $d\omega=\theta\wedge\omega$ on $M$.

\end{enumerate}
\end{prop}

\begin{proof}
Let $\lambda\in C^\infty(M)$ such that $0\le\lambda\le1$, $\lambda=1$ on $T$, and $\supp\lambda\subset T_0$, and let $\lambda'\in\Cinftyc(M')$ such that $\supp\lambda\subset T'_0$ and $\lambda'\equiv\lambda$ on $T'_0\equiv T_0$.

To prove~(\ref{i: Z'}), let $Z'_0\equiv Z$ on $T'_0\equiv T_0$, and take $Z'\equiv Y'+\lambda'(Z'_0-Y')$. 

To prove~(\ref{i: A -> B}), let $B_0\equiv A$ on $T_0\equiv T'_0$, and take $B=Z+\lambda(B_0-Z)$.

To prove~(\ref{i: omega, theta, rho}), take $\omega_0\equiv\omega'$ and $\theta_0\equiv\theta'$ on $T_0$. Take $\omega_1\in C^\infty(M;\Lambda^1)$ defining $\FF$. Then $\omega=\lambda\omega_0+(1-\lambda)\omega_1$ also defines $\FF$. Thus $d\omega=\theta_1\wedge\omega$ for some $\theta_1\in C^\infty(M;\Lambda^1)$. We get $(\theta_0-\theta_1)\wedge\omega=0$ on $T$, and therefore~(\ref{i: omega, theta, rho}) is satisfied $\theta=\theta_1+\lambda(\theta_0-\theta_1)$.
\end{proof}

We can also consider a transverse power change of the differential structure on every $M'_L$ around $L$ (Section~\ref{ss: change of diff struct}). The corresponding new differentiable structure on every $T_L\equiv T'_L$ can be combined with the differentiable structure of $M^1$ to produce a new differentiable structure on $M$, also called a \emph{transverse power change} of the differentiable structure (around $M^0$), and keeping $Z\in\fX(M,\FF)$ after this change. In this way, the absolute values $|\varkappa_L|$ can be changed arbitrarily, but keeping every $\sign(\varkappa_L)$ invariant.

Consider the forms $\omega$ and $\theta$ of Proposition~\ref{p: Z', A}~(\ref{i: omega, theta, rho}), and let $\rho\equiv\rho'$ on $T\equiv T'$. With the new differentiable structure, $C^\infty(T)$ is generated by $\rho_\alpha:=\rho\,|\rho|^{\alpha-1}$ and $C^\infty(M^0)\equiv\varpi^*C^\infty(M^0)$. Moreover $\omega_\alpha:=\rho^{\alpha-1}\omega$ and $\theta_\alpha:=\alpha\theta$ have smooth extensions to $T$, $\omega_\alpha$ defines $\FF|_T$, $d\omega_\alpha=\theta_\alpha\wedge\omega_\alpha$ on $T$, and $d\rho_\alpha=\rho_\alpha\theta_\alpha+|\alpha\varkappa_L|\,\omega_\alpha$ on $T_L$. Like in Proposition~\ref{p: Z', A}~(\ref{i: omega, theta, rho}), the restrictions of $\omega_\alpha$ and $\theta_\alpha$ to some smaller tubular neighborhood of $L$ can be extended to $M$, keeping the relation $d\omega_\alpha=\theta_\alpha\wedge\omega_\alpha$.

\subsection{Transverse structure}\label{ss: transv struct}

Let $\PP$ be the pseudogroup on $S^1_\infty=\R\cup\{\infty\}$ generated by the projective rotation $x\mapsto-1/x$, the hyperbolic projective transformations $x\mapsto\lambda x$ ($\lambda>0$), and the diffeomorphisms $x\mapsto x^\alpha$ of $\R^+$ ($\alpha>0$). $\FF$ is called a \emph{$\PP$-foliation} if $\{U_k,x_k\}$ can be chosen such that every $\Sigma_k$ is realized as an open subset of $S^1_\infty$ and the maps $h_{kl}$ belong to $\PP$.

\begin{prop}\label{p: PP-foln}
$\FF$ is a $\PP$-foliation.
\end{prop}

\begin{proof}
Since $\FF\equiv\FF'$ on every $T_{L,0}\equiv T'_{L,0}$ (Section~\ref{ss: tubular neighborhoods}), the restriction of $\FF$ to any $T_{L,0}$ has a regular foliated atlas $\{U_a,(x_a,y_a)\}$ such that the corresponding elementary holonomy transformations are restrictions of homotheties. For $\varkappa=\varkappa_L\in\R^\times$, we have $\overline Z=\varkappa x_a\,\partial_{x_a}$ on $x_a(U_a)$ (Section~\ref{ss: tubular neighborhoods}), whose local flow $\bar\phi_a$ is given by $\bar\phi_a(x,t)=e^{\varkappa t}x$.

Now the restrictions $\FF^1_l$ are $\R$-Lie foliations according to Sections~\ref{s: transv simple foliated flows}. Then any $\FF^1_l$ has a regular foliated atlas $\{V_i,(w_i,v_i)\}$ whose elementary holonomy transformations are given by translations, $w_i=h_{ij}(w_j)=w_j+c_{ij}$, between open intervals of $\R$. Taking the new transverse coordinates $u_i=e^{w_i}$, we get another regular foliated atlas $\{V_i,(u_i,v_i)\}$ of $\FF^1_l$, whose elementary holonomy transformations are given by homotheties, $u_i=e^{c_{ij}}u_j$, between open intervals of $\R^+$. Thus $\{V_i,(u_i,v_i)\}$ defines a transversely affine structure of $\FF^1_l$. With the notation of Section~\ref{s: transv simple foliated flows}, we can indeed assume that $\pi_l:\widetilde V_i\to V_i$ is a diffeomorphism for some open $\widetilde V_i\subset\widetilde M^1_l$, and $u_i\pi_l=D_l$ on $\widetilde V_i$. Hence $\overline Z=\partial_{w_i}$ on $w_i(V_i)$, and therefore $\overline Z=u_i\partial_{u_i}$ on $u_i(V_i)$, whose local flow $\bar\phi_i$ is given by $\bar\phi_i^t(u)=e^tu$.

For any nonempty intersection $U_a\cap V_i$, via the corresponding elementary holonomy transformation $h_{ai}=x_au_i^{-1}$, the vector field $\varkappa x_a\,\partial_{x_a}$ corresponds to $u_i\partial_{u_i}$, and therefore $\bar\phi_a$ corresponds to $\bar\phi_i$. Take any $p\in U_a\cap V_i$, and let $\bar p_a=x_a(p)\in\R^\times$ and $\bar p_i=u_i(p)\in\R^+$. Then, for $|t|$ small enough,
\[
h_{ai}(e^t\bar p_i)=h_{ai}\bar\phi_i^t(\bar p_i)=\bar\phi_a^t(\bar p_a)=e^{\varkappa t}\bar p_a=\bar p_a\bar p_i^{-\varkappa}(e^t\bar p_i)^\varkappa\;,
\]
yielding $h_{ai}(u)=\bar p_a\bar p_i^{-\varkappa}u^\varkappa$ for u close enough to $\bar p_i$. Since $h_{ai}$ preserves the orientation, $\bar p_a$ and $\varkappa$ must have the same sign. Then $h_{ai}$ can be expressed as a composition of generators of $\PP$: 
\begin{alignat}{2}
u&\mapsto\tilde u:=u^\varkappa\mapsto \bar p_a\bar p_i^{-\varkappa}\tilde u&\qquad\text{if}\ \bar p_a,\varkappa>0\;,\label{u to tilde u to ...}\\
u&\mapsto\tilde u:=u^{-\varkappa}\mapsto\hat u:=|\bar p_a|^{-1}\bar p_i^\varkappa\tilde u\mapsto-1/\hat u&\qquad\text{if}\ \bar p_a,\varkappa<0\;.\label{u to tilde u to hat u to ...}
\end{alignat}
Thus a union of foliated atlases of these types, for all  $L\in\pi_0M^0$ and foliations $\FF_l$, is a foliated atlas of $\FF$ defining a structure of $\PP$-foliation.
\end{proof}

\begin{prop}\label{p: transv projective}
After performing some transverse power change of the differentiable structure around $M^0$, $\FF$ becomes transversely projective.
\end{prop}

\begin{proof}
Using a transverse power change of the differentiable structure around $M^0$, we can assume that $\varkappa_L=\pm1$ for all $L\in\pi_0M^0$. Then, in the proof of Proposition~\ref{p: PP-foln}, the elementary holonomy transformations~\eqref{u to tilde u to ...} and~\eqref{u to tilde u to hat u to ...} are also restrictions of elements of $\PSL(2,\R)$.
\end{proof}

\section{Existence and description of simple foliated flows}\label{s: existence}

Now let $\FF$ be any smooth transversely oriented foliation of codimension one on a closed manifold $M$. 

\subsection{Existence of simple foliated flows}\label{ss: existence}

\begin{prop}\label{p: simple preserved leaves => there exist simple flows}
If $(M,\FF)$ admits some transversely simple foliated flow $\phi$, then it also admits some simple foliated flow $\psi$ with $\bar\phi=\bar\psi$.
\end{prop}

\begin{proof}
Let $Z\in\fX(M,\FF)$ be the infinitesimal generator of $\phi$, and consider the notation of Section~\ref{ss: tubular neighborhoods}. Take some simple flow $\zeta$ on $M^0$ without closed orbits (Example~\ref{ex: Morse}), and let $A$ denote its infinitesimal generator. By Proposition~\ref{p: A -> B}, there is some simple $B\in\fXcom(M',\FF')$, without closed orbits, such that $B|_{M^0}=A$ and $\overline{B}=\overline{Z'}$. Then, by Proposition~\ref{p: Z', A}~(\ref{i: A -> B}), there is some $C\in\fX(M,\FF)$ with $\overline{C}=\overline{Z}$, $C\equiv B$ on $T\equiv T'$, and $C=Z$ on $M\sm T_0$. 

By Peixoto's extension to open manifolds of a theorem of Kupka and Smale (Section~\ref{ss: simple flows}), there is some generic $D\in\fX(M^1)$ as close as desired to $C|_{M^1}$ in the strong $C^\infty$ topology; in particular, $D$ is simple. If $D$ close enough to $C|_{M^1}$ in the strong $C^\infty$ topology, then $D$ has an extension $E\in\fX(M)$ with $E|_{M^0}=A$, and $\overline C=f\overline{E}$ in $C^\infty(M;N\FF)$ for some $0<f\in C^\infty(M)$ with $f=1$ on $M^0$. Thus $fE\in\fX(M,\FF)$ and $\overline{fE}=\overline Z$, and therefore the foliated flow $\psi$ of $fE$ satisfies $\bar\psi=\bar\phi$. So $\psi$ is transversely simple and has the same preserved leaves as $\phi$ (the leaves in $M^0$); in particular, $\psi$ has no fixed points in $M^1$. Since $fE=E=C\equiv B=A$ on $M^0$, we get that $\psi$ agrees with $\zeta$ on $M^0$, and therefore its fixed points are simple by Proposition~\ref{p: phi_0 simple => phi simple in M^0}. Moreover $fE|_{M^1}=fD$ is simple by Remark~\ref{r: fZ}.
\end{proof}

\begin{defn}\label{d: weakly simple foliated flow}
  It is said that $\phi$ {\rm(}or $Z${\rm)} is {\em weakly simple\/} if its preserved leaves are transversely simple and its closed orbits are simple.
\end{defn}

By Proposition~\ref{p: phi_0 simple => phi simple in M^0}, simple foliated flows are weakly simple.

\begin{prop}\label{p: transversely simple => there exist weakly simple flows equal to id on M^0}
If $(M,\FF)$ has some transversely simple foliated flow $\phi$, then it also has some weakly simple foliated flow $\zeta$ such that $\bar\phi=\bar\zeta$, $\zeta^t=\id$ on $M^0$ for all $t$, and $\zeta$ has no closed orbit in some neighborhood of $M^0$.
\end{prop}

\begin{proof}
Apply Proposition~\ref{p: Z', A}~(\ref{i: A -> B}) with some transversely simple $Z\in\fX(M,\FF)$ and $A=Y'$.
\end{proof}

\subsection{Description of foliations with simple foliated flows}\label{ss: description}

Now, without requiring the existence of any special foliated flow a priori, assume that $\FF$ satisfies the following properties:
\begin{enumerate}[(A)]

\item\label{i: almost w/o hol} $\FF$ is almost without holonomy with finitely many leaves with holonomy.

\item\label{i: homotheties} The holonomy groups of the compact leaves can be described as groups of germs at $0$ of homotheties on $\R$.

\end{enumerate}
By~(\ref{i: almost w/o hol}), we can use the notation of Section~\ref{ss: folns almost w/o hol}. In the following, we refer to the possibilities~(\ref{i: fiber bundle})--(\ref{i: all folns FF_l are models (2)}) of Section~\ref{s: transv simple foliated flows} for transversely simple flows.

\begin{ex}\label{ex: fiber bundle} 
Suppose that $\FF$ is given by a fiber bundle $M\to S^1$ with connected fibers. For any even number of points, $x_1,\dots,x_{2m}\in S^1$ ($m\ge0$), in cyclic order, and numbers $\varkappa_1,\dots,\varkappa_{2m}\in\R^\times$, with alternate sign, there is some simple flow $\bar\phi$ on $S^1$ such that $\Fix(\bar\phi)=\{x_1,\dots,x_{2m}\}$ and $\bar\phi^t_*=e^{\kappa_jt}$ on $T_{x_j}S^1\equiv\R$. By Proposition~\ref{p: simple preserved leaves => there exist simple flows}, there is a simple foliated flow $\phi$ on $(M,\FF)$ whose preserved leaves are fibers $L_1,\dots,L_{2m}$ over $x_1,\dots,x_{2m}$. If $m>0$, then $\phi$ has no closed orbits in $M^1$. If $m=0$, then $\phi$ has no preserved leaves, and therefore no fixed points. This is of type~(\ref{i: fiber bundle}).
\end{ex}

\begin{ex}\label{ex: minimal Lie foln} 
If $\FF$ is an $\R$-Lie foliation with dense leaves, $\olfX(M,\FF)$ is of dimension $1$ and generated by a non-vanishing transverse vector field. Hence there are simple foliated flows by Proposition~\ref{p: simple preserved leaves => there exist simple flows}, all of them without preserved leaves. This is  of type~(\ref{i: minimal Lie foln}),
\end{ex}

\begin{ex}\label{ex: transv affine foln} 
Suppose that $\FF$ is a transversely affine foliation that is not an $\R$-Lie foliation. Then, according to Section~\ref{ss: transv affine folns}, to get~(\ref{i: almost w/o hol}), $\FF$ is elementary, and we can assume that $\im D=\R$ and $\Hol\FF$ is a non-trivial group of homotheties. Then, by Lemma~\ref{l: Z = varkappa x partial_x}~(\ref{i: Z = varkappa x partial_x}) and Proposition~\ref{p: olfX(M, FF) equiv fX(im D, Gamma), transversely affine}, $\olfX(M,\FF)$ is generated by a transverse vector field $\overline Z$ such that the foliated flow $\phi$ of $Z$ is transversely simple.  By Proposition~\ref{p: simple preserved leaves => there exist simple flows}, there is a simple foliated flow $\phi'$ with $\bar\phi'=\bar\phi$. It also follows from Lemma~\ref{l: Z = varkappa x partial_x}~(\ref{i: Z = varkappa x partial_x}) and Proposition~\ref{p: olfX(M, FF) equiv fX(im D, Gamma), transversely affine} that there is some $\varkappa\in\R^\times$ such that $\{\,\varkappa_L\mid L\in\pi_0M^0\,\}=\{\varkappa\}$.
\end{ex}

\begin{ex}\label{ex: transv proj foln} 
Assume that $\FF$ is a transversely projective foliation that is not transversely affine. Then, according to Section~\ref{ss: transv proj folns}, to get~(\ref{i: almost w/o hol}) and~(\ref{i: homotheties}), we can assume that $\im D=S^1_\infty$ and $\Hol\FF$  consists of the identity and hyperbolic elements with common fixed point set $\{0,\infty\}$ and possible elliptic elements that keep $\{0,\infty\}$ invariant. By Lemma~\ref{l: Z = varkappa x partial_x, transv proj} and the projective version of Proposition~\ref{p: olfX(M, FF) equiv fX(im D, Gamma), transversely affine}, to get $\olfX(M,\FF)\ne0$, there must be no elliptic element in $\Hol\FF$. Moreover, in this case, $\olfX(M,\FF)$ is generated by a transverse vector field $\overline Z$ such that the foliated flow $\phi$ of $Z$ is transversely simple. By Proposition~\ref{p: simple preserved leaves => there exist simple flows}, there is some simple foliated flow $\phi'$ with $\bar\phi'=\bar\phi$. By Lemma~\ref{l: Z = varkappa x partial_x, transv proj} and the projective version of Proposition~\ref{p: olfX(M, FF) equiv fX(im D, Gamma), transversely affine}, there is some $\varkappa\in\R^+$ such that $\{\,\varkappa_L\mid L\in\pi_0M^0\,\}=\{\pm\varkappa\}$.
\end{ex}

\begin{ex}\label{ex: transv change of diff str} In Examples~\ref{ex: transv affine foln} and~\ref{ex: transv proj suspension foln}, we can consider any transverse power change of the differentiable structure around $M^0$ (Sections~\ref{ss: change of diff struct} and~\ref{ss: tubular neighborhoods}). With the new differentiable structure, the foliation has the same simple foliated flows, but the absolute values $|\varkappa_L|$ can be arbitrary, keeping the same signs $\sign(\varkappa_L)$. Thus $\{\,\sign(\varkappa_L)\mid L\in\pi_0M^0\,\}$ is $\{1\}$ or $\{\pm1\}$ if and only we have changed the differential structure of Example~\ref{ex: transv affine foln} or~\ref{ex: transv proj suspension foln}, respectively.
\end{ex}

Examples~\ref{ex: transv affine foln}--\ref{ex: transv change of diff str} can be of type~(\ref{i: all folns FF_l are models (1)}) or~(\ref{i: all folns FF_l are models (2)}).

\begin{thm}\label{t: simple foliated flows}
For any smooth transversely oriented foliation of codimension one on a closed manifold, the following conditions are equivalent:
\begin{enumerate}[(i)]

\item\label{i: A-C} It satisfies~(\ref{i: almost w/o hol}) and~(\ref{i: homotheties}).

\item\label{i: exs} It is described by one of Examples~\ref{ex: fiber bundle}--\ref{ex: transv change of diff str}.

\item\label{i: preserved leaves are simple} It admits a transversely simple foliated flow.

\item\label{i: weakly simple foliated flow trivial on preserved leaves} It admits a weakly simple foliated flow (trivial on its preserved leaves).

\item\label{i: simple foliated flow} It admits a simple foliated flow.

\end{enumerate}
\end{thm}

\begin{proof}
We already know that~(\ref{i: preserved leaves are simple}) yields~(\ref{i: A-C}) (Section~\ref{s: transv simple foliated flows}). By Proposition~\ref{p: transv projective}, Examples~\ref{ex: fiber bundle}--\ref{ex: transv change of diff str} cover all cases~(\ref{i: fiber bundle})--(\ref{i: all folns FF_l are models (2)}), and therefore~(\ref{i: A-C}) yields~(\ref{i: exs}). Proposition~\ref{p: simple preserved leaves => there exist simple flows} states that~(\ref{i: preserved leaves are simple}) yields~(\ref{i: simple foliated flow}), which was used in Examples~\ref{ex: fiber bundle}--\ref{ex: transv change of diff str}, showing that~(\ref{i: exs}) yields~(\ref{i: simple foliated flow}). Proposition~\ref{p: transversely simple => there exist weakly simple flows equal to id on M^0} states that~(\ref{i: preserved leaves are simple}) yields~(\ref{i: weakly simple foliated flow trivial on preserved leaves}). The remaining implications are obvious.
\end{proof}

According to Theorem~\ref{t: simple foliated flows}, the foliations of Examples~\ref{ex: Kronecker's flow},~\ref{ex: trasv aff foln on S^n-1 times S^1},~\ref{ex: connected sum of transv affine folns} and~\ref{ex: transv proj suspension foln}--\ref{ex: connected sum of transv proj folns} admit simple foliated flows.

\bibliographystyle{amsplain}


\providecommand{\bysame}{\leavevmode\hbox to3em{\hrulefill}\thinspace}
\providecommand{\MR}{\relax\ifhmode\unskip\space\fi MR }
\providecommand{\MRhref}[2]{%
  \href{http://www.ams.org/mathscinet-getitem?mr=#1}{#2}
}
\providecommand{\href}[2]{#2}

\end{document}